\documentclass[11pt]{amsart}
\usepackage{amsmath,amssymb}
\usepackage{geometry}
\usepackage{graphicx}
\usepackage{verbatim}
\usepackage{color}
\usepackage{hyperref}
\usepackage{epstopdf}

\begin{document}

\newtheorem{thm}{Theorem}[section]
\newtheorem{theorem}{Theorem}[section]
\newtheorem{lem}[thm]{Lemma}
\newtheorem{lemma}[thm]{Lemma}
\newtheorem{prop}[thm]{Proposition}
\newtheorem{proposition}[thm]{Proposition}
\newtheorem{corollary}[thm]{Corollary}
\newtheorem{definition}[thm]{Definition}
\newtheorem{remark}[thm]{Remark}
\newtheorem{conjecture}[theorem]{Conjecture}

\numberwithin{equation}{section}

\newcommand{\Z}{{\mathbb Z}} 
\newcommand{\Q}{{\mathbb Q}}
\newcommand{\R}{{\mathbb R}}
\newcommand{\C}{{\mathbb C}}
\newcommand{\N}{{\mathbb N}}
\newcommand{\FF}{{\mathbb F}}
\newcommand{\fq}{\mathbb{F}_q}
\newcommand{\rmk}[1]{\footnote{{\bf Comment:} #1}}

\newcommand{\bfA}{{\boldsymbol{A}}}
\newcommand{\bfY}{{\boldsymbol{Y}}}
\newcommand{\bfX}{{\boldsymbol{X}}}
\newcommand{\bfZ}{{\boldsymbol{Z}}}
\newcommand{\bfa}{{\boldsymbol{a}}}
\newcommand{\bfy}{{\boldsymbol{y}}}
\newcommand{\bfx}{{\boldsymbol{x}}}
\newcommand{\bfz}{{\boldsymbol{z}}}
\newcommand{\F}{\mathcal{F}}
\newcommand{\Gal}{\mathrm{Gal}}
\newcommand{\Fr}{\mathrm{Fr}}
\newcommand{\Hom}{\mathrm{Hom}}
\newcommand{\GL}{\mathrm{GL}}

\renewcommand{\mod}{\;\operatorname{mod}}
\newcommand{\ord}{\operatorname{ord}}
\newcommand{\TT}{\mathbb{T}}
\renewcommand{\i}{{\mathrm{i}}}
\renewcommand{\d}{{\mathrm{d}}}
\renewcommand{\^}{\widehat}
\newcommand{\HH}{\mathbb H}
\newcommand{\Vol}{\operatorname{vol}}
\newcommand{\area}{\operatorname{area}}
\newcommand{\tr}{\operatorname{tr}}
\newcommand{\norm}{\mathcal N} 
\newcommand{\intinf}{\int_{-\infty}^\infty}
\newcommand{\ave}[1]{\left\langle#1\right\rangle} 
\newcommand{\Var}{\operatorname{Var}}
\newcommand{\Prob}{\operatorname{Prob}}
\newcommand{\sym}{\operatorname{Sym}}
\newcommand{\disc}{\operatorname{disc}}
\newcommand{\CA}{{\mathcal C}_A}
\newcommand{\cond}{\operatorname{cond}} 
\newcommand{\lcm}{\operatorname{lcm}}
\newcommand{\Kl}{\operatorname{Kl}} 
\newcommand{\leg}[2]{\left( \frac{#1}{#2} \right)}  
\newcommand{\Li}{\operatorname{Li}}

\newcommand{\sumstar}{\sideset \and^{*} \to \sum}

\newcommand{\LL}{\mathcal L} 
\newcommand{\sumf}{\sum^\flat}
\newcommand{\Hgev}{\mathcal H_{2g+2,q}}
\newcommand{\USp}{\operatorname{USp}}
\newcommand{\conv}{*}
\newcommand{\dist} {\operatorname{dist}}
\newcommand{\CF}{c_0} 
\newcommand{\kerp}{\mathcal K}

\newcommand{\Cov}{\operatorname{cov}}
\newcommand{\Sym}{\operatorname{Sym}}

\newcommand{\ES}{\mathcal S} 
\newcommand{\EN}{\mathcal N} 
\newcommand{\EM}{\mathcal M} 
\newcommand{\Sc}{\operatorname{Sc}} 
\newcommand{\Ht}{\operatorname{Ht}}

\newcommand{\E}{\operatorname{E}} 
\newcommand{\sign}{\operatorname{sign}} 

\newcommand{\divid}{d} 

\newcommand{\h}{\mathcal{H}_{2g+1,q}}
\newcommand{\p}{\mathcal{P}_{2g+1,q}}
\newcommand{\f}{\mathbb{F}_{q}[T]}
\newcommand{\z}{\zeta_A}
\newcommand{\lo}{\log_q}
\newcommand{\x}{\chi}
\newcommand{\xx}{\mathcal{X}}
\newcommand{\lL}{\mathcal{L}}
\newcommand{\e}{\varepsilon}

\title[Moments of Dirichlet $L$-functions]
{On the moments of certain families of Dirichlet $L$-functions}

\author{J. C. Andrade}
\address{Department of Mathematics, University of Exeter, Exeter, EX4 4QF, United Kingdom}
\email{j.c.andrade@exeter.ac.uk}

\author{K. Smith}
\address{Department of Mathematics, University of Exeter, Exeter, EX4 4QF, United Kingdom}
\email{ks614@exeter.ac.uk}


\begin{abstract}
In this paper we address the problem of computing asymptotic formulae for the expected values and second moments of central values of primitive Dirichlet $L$-functions $L(1/2,\chi_{8d}\otimes\psi)$ when $\psi$ is a fixed even primitive non-quadratic character of odd modulus $q$, $\chi_{8d}$ is a primitive quadratic character, $d\equiv h\pmod r$ is odd and squarefree and $r\equiv0\pmod q$ is even. Restricting to these arithmetic progressions ensures that the resulting sets of $L$-functions each form a ``family of primitive $L$-functions" in the specific sense defined by Conrey, Farmer, Keating, Rubinstein and Snaith. Soundararajan had previously computed these quantities without restricting them to arithmetic progressions. It turns out that the restriction to arithmetic progressions introduces non-diagonal terms that require significantly more detailed analysis which we carry on in this paper.
\end{abstract}

\date{\today}

\maketitle


\section{Introduction and statement of results}

The moments of families of primitive $L$-functions are central objects of study in analytic number theory due to their many applications and the rich conjectural structure that relates them to the spectral theory of random matrices in the classical compact groups. Associated with each primitive $L$-function 
\begin{eqnarray}\label{dirl}
L(s,f)=\sum_{n=1}^{\infty}\frac{f(n)}{n^s},
\end{eqnarray}
the $L$-function
\begin{eqnarray}\label{dirl}
L(s,f\otimes\chi_d)=\sum_{n=1}^{\infty}\frac{f(n)\chi_{d}(n)}{n^s}
\end{eqnarray}
with $\chi_d$ a primitive quadratic Dirichlet character (so $d$ is a fundamental discriminant and $\chi_d(\cdot)=(d/\cdot)$ is the Kronecker symbol), is a quadratic twist of $L(s,f)$. 
According to the conjectural classification of Katz and Sarnak \cite{KS}, sets of quadratic twists with $|d|\leq X$ and $X\rightarrow\infty$ form unitary, orthogonal or symplectic families (depending on the coefficients $f(n)$ of the original $L$-function) and, by analogy with the distribution of eigenvalues of large random matrices in the corresponding group, this classification is expected to govern the distribution of zeros in the family near the central point. \\

The values of $L$-functions at the central point are of special interest. In particular, it is conjectured that Dirichlet $L$-functions do not vanish at their central point $s=1/2$. In this direction, for a fixed non-quadratic primitive Dirichlet character  $\psi$ of odd modulus $q$ and $8d$ ranging over all positive fundamental discriminants divisible by $8$ and coprime with $q$, Soundararajan \cite{S1} showed that at least one fifth of the family of Dirichlet $L$-functions $L(s,\psi\otimes\chi_{8d})$ are non-zero at $s=1/2$. He achieved this by computing the mollified moments, which firstly require asymptotic formulae (with good uniformity in the parameter $l$) for the expected values 
\begin{eqnarray}\label{mom1}
\frac{1}{X}\sum_{(d,2q)=1\textrm{ squarefree}}\chi_{8d}\left(l\right)L(1/2,\psi\otimes\chi_{8d})\Phi\left(\frac{d}{X}\right)
\end{eqnarray}
and the second moments
\begin{eqnarray}\label{mom2}
\frac{1}{X}\sum_{(d,2q)=1\textrm{ squarefree}}\chi_{8d}\left(l\right)\left|L(1/2,\psi\otimes\chi_{8d})\right|^2\Phi\left(\frac{d}{X}\right),
\end{eqnarray}
where $\Phi\leq 1$ is a fixed smooth positive compactly supported approximation to the characteristic function of the interval $(1,2)$. Soundararajan found asymptotic formulae for the expected values (\ref{mom1}) using the Polya-Vinogradov inequality, and for the second moments (\ref{mom2}) using a version of the Poisson summation formula. A key step of his argument is the demonstration that the resulting non-diagonal terms are negligible, while he noted that this is not the case when $\psi$ is quadratic. In later work \cite{S2}, Soundararajan succeeded to isolate and evaluate the non-diagonal contribution in the case $\psi=1$, allowing him to compute asymptotic formulae for the first three moments and to show that at least $7/8$ of this family are non-zero at $s=1/2$; a key step of the argument being a functional equation that facilitates the evaluation of a certain integral representation of the non-diagonal main terms.  \\

Based on a conjectural ``recipe''  and a certain combinatorial mechanism (\cite{CFKRS}, Lemma 2.5.2) that casts the integral moments of sets of quadratic twists of primitive of $L$-functions with real coefficients into the same form as the moments of the characteristic polynomials for the orthogonal or symplectic groups, Conrey, Farmer, Keating, Rubinstein and Snaith have conjectured asymptotic formulae for all the integral moments of specific families of quadratic twists of primitive $L$-functions. These families are parametrised by those fundamental discriminants that lie in arithmetic progressions for which the parameters in the functional equations  of the twisted $L$-functions are invariants of the conductor (i.e. the root numbers and the spectral parameters). The family of quadratic Dirichlet $L$-functions studied by Soundararajan in \cite{S2} is an example of the symplectic case (at least, he showed that the first three moments match the  symplectic statistics). However, the families of $L$-functions studied by Soundararajan in \cite{S1} are not quite families in the specific sense of CFKRS because the parameters in the functional equations of the $L(s,\psi\chi_{8d})$ are not invariant as $8d$ ranges over positive fundamental discriminants coprime with $q$. Restricting to positive fundamental discriminants coprime with $q$ in a fixed residue class (mod $r$) with $r\equiv 0\pmod q$ remedies this. For example, the subsets
\begin{eqnarray}\label{families}
\mathcal{F}_{h,r}(\psi)=\left\{ L\left(s,\psi\otimes\chi_{8d}\right): d\equiv h  \textrm{ (mod $r$) is odd and squarefree}\right\}\nonumber
\end{eqnarray}
with $r\equiv 0\pmod q$ even and squarefree, $(h,r)=1$, are families of primitive $L$-functions in the sense of CFKRS. \\

In this paper we compute asymptotic expansions for the analogues of (\ref{mom1}) and (\ref{mom2}) over the families $\mathcal{F}_{h,r}(\psi)$. With a view to introducing mollifiers later, we pay careful attention to obtaining good uniformity in the parameter $l$. Our main results are as follows.

\begin{theorem}\label{th1}Firstly, if $\psi$ is a fixed even primitive non-quadratic character of odd modulus $q$, $r\equiv 0\pmod q$ is even and squarefree, $(h,r)=1$ and $lrY^2\ll X^{1/2-\delta}$ for some fixed $\delta>0$, then there is a fixed $C_{h,l,r,\psi,\Phi}$ such that
\begin{eqnarray}\label{mom3}
\frac{1}{X}\sum_{\substack{d\equiv h \pmod r\\d\textrm{ \emph{squarefree}}}}\chi_{8d}\left(l\right)L(1/2,\psi\otimes\chi_{8d})\Phi\left(\frac{d}{X}\right)=C_{h,l,r,\psi,\Phi}+E_{l,r,\Phi,X,Y}
\end{eqnarray}
where 
\begin{eqnarray}\label{th1errorbound}
E_{l,r,\Phi,X,Y} \ll_{\epsilon,\Phi} \frac{l^{3/2}rY^2}{X^{1/2-\epsilon}} +\frac{l^{1/2}r^{\epsilon-1}}{X^{1/4-\epsilon}} +\frac{X^{\epsilon}}{Y^{1/2}}
\end{eqnarray}
for any fixed $\epsilon>0$. \\

Secondly, if $\psi=1$ is the trivial character, there is a polynomial $P_{h,l,r,\Phi}$ of degree $1$ such that
\begin{eqnarray}\label{mom3b}
\frac{1}{X}\sum_{\substack{d\equiv h \pmod r\\d\textrm{ \emph{squarefree}}}}\chi_{8d}\left(l\right)L(1/2,\chi_{8d})\Phi\left(\frac{d}{X}\right)=P_{h,l,r,\Phi}(\log X)+E_{l,r,\Phi,X,Y}.
\end{eqnarray}
Moreover, if $(l,r)>1$ then 
\begin{eqnarray}\label{zeros}
C_{h,l,r,\psi,\Phi}=P_{h,l,r,\Phi}=0.
\end{eqnarray} 
\end{theorem}

\begin{theorem}\label{th2}If $\psi$ is a fixed even primitive non-quadratic character of odd modulus $q$ then, under the same conditions on $h,l,r$ and $Y$ as in Theorem \ref{th1}, there is a polynomial $\mathcal{P}_{h,l,r,\psi,\Phi}$ of degree $1$ such that 
\begin{eqnarray}\label{mom4}
\frac{1}{X}\sum_{\substack{d\equiv h \pmod r\\d\textrm{ \emph{squarefree}}}}\chi_{8d}\left(l\right)\left|L(1/2,\psi\otimes\chi_{8d})\right|^2\Phi\left(\frac{d}{X}\right)=\mathcal{P}_{h,l,r,\psi,\Phi}(\log X)+\mathcal{E}_{l,r,\Phi,X,Y}\nonumber\\
\end{eqnarray}
where 
\begin{eqnarray}\label{th2errorbound}
\mathcal{E}_{l,r,\Phi,X,Y} \ll_{\epsilon,\Phi} \frac{l^{\epsilon-1/2}r^{\epsilon-1}}{X^{1/2-\epsilon}} +\frac{l^{1+\epsilon}r^{2+\epsilon}Y^{1+\epsilon}}{X^{3/8-\epsilon}} +\frac{X^{\epsilon}}{Y^{}}+\frac{l^{\epsilon}r^{\epsilon-2}}{Y^{1-\epsilon}}
\end{eqnarray}
and $\mathcal{P}_{h,l,r,\psi,\Phi}=0$ if $(l,r)>1$.
\end{theorem}
Theorems \ref{th1} and \ref{th2} are proved in sections \ref{ev} and \ref{sm} respectively. We broadly follow the methodology of Soundararajan's later work \cite{S2}, but we use a version of the Poisson summation formula that picks out the arithmetic progression that is also used by Radziwill and Soundararajan in \cite{SR} (actually we prove a more general version in Section \ref{pois}). The non-diagonal terms arising from the Poisson summation formula involve certain trigonometric functions on the group $\mathbb{Z}^{\times}_r$. Trivial bounds suffice to show that their contribution to the expected values is negligible, which is done in Section \ref{bn}. To handle the second moments, in Section \ref{exp} we expand the trigonometric functions as linear combinations of Dirichlet characters (mod $r$) which reveals non-negligible non-diagonal contributions. Specifically, the non-diagonal terms contribute constant terms $\mathcal{N}_{h,l,r,\psi,\Phi}\ll r^{\epsilon-2}$  to the polynomials $\mathcal{P}_{h,l,r,\psi,\Phi}$, and we note that the expected value of $\mathcal{N}_{h,l,r,\psi,\Phi}$ over the group $h\in\mathbb{Z}^{\times}_r$ is always zero. For instance, when $r=2q$, this expected value recovers the second moments of the families studied by Soundararajan. In the course of our proof of Theorem \ref{th2} we derive an integral representation for the non-diagonal main terms. In particular, in Section \ref{quarticderiv} we show that when $\psi$ is a quartic character the integral evaluates as 
\begin{eqnarray}\label{quarticterm}
\mathcal{N}_{h,1,2q,\psi,\Phi}&=&\frac{\widehat\Phi(0)L\left(1, \varphi_0\psi^2 \right) \Re\tau\left(\overline\psi\right)}{q^2\left(1\pm q\right)}
\nonumber\\
&\times&\left(\mathcal{K}_{\psi^2}'(0;1,2q)+\left(\gamma+\frac{2\Gamma'\left(\frac{1}{4}\right)}{\Gamma\left(\frac{1}{4}\right)}+\log\left(\frac{8q}{\pi^2}\right)+\frac{\pi}{2}\right)\mathcal{K}_{\psi^2}(0;1,2q)\right)\nonumber
\end{eqnarray}
where $\varphi_0$ is the principal character (mod $r$), $\tau$ is the Gauss sum, `$\pm$' depends on whether $q$ is congruent to $1$ or $3$ (mod $4$) and
\begin{eqnarray}\label{keuler6}
\mathcal{K}_{\psi^2}(s;1,2q)&=&2^{1-2s}\psi^2(2)L\left(2s,\psi^2\right)L\left(2s+1,\psi^2\right)\left(1-\frac{\psi^2(2)}{2^{1+2s}}\right)\left(1-\frac{\psi^2(2)}{2^{1-2s}}\right)\nonumber\\
&\times &\prod_{p\textrm{ odd}}\left(1-\frac{1}{p^{}}\right)\left(1-\frac{\psi^2(p)}{p^{}}\right)\left(1+\frac{1+\psi^2(p)}{p} +\frac{1}{p^3} -\frac{\psi^2(p)}{p^2}\left(\frac{1}{p^{2s}}+p^{2s}    \right)   \right).\nonumber
\end{eqnarray}
Here the situation is somewhat analogous to the case $\psi=1$ studied by Soundararajan; the integral representation of $\mathcal{N}_{h,1,2q,\psi,\Phi}$ may be evaluated because the integrand satisfies an analogous functional equation in this case. It is not clear whether such functional equations exist otherwise. \\

Lastly we mention that, according to the philosophy of Katz and Sarnak \cite{KS}, the statistics of the low-lying zeros in the families $\mathcal{F}_{h,r}(\psi)$ follow a unitary law which lacks repulsion from $s=1/2$ and the degree of the polynomial in the asymptotic formula for the $2k$th moment is expected to be $k^2$ (in contrast to the symplectic case studied by Soundararajan in which the degree of the $k$th moment is expected to be $k(k+1)/2$). Yet, the combinatorial mechanism (\cite{CFKRS}, Lemma 2.5.2) is inapplicable in this case because $\psi$ is not a real character. It would be interesting to extend the CFKRS model to include these families. Another interesting question is whether or not the fourth moment can be computed unconditionally.

\section{Auxiliary results}
\subsection{Poisson summation}\label{pois}
In this work we must handle sums of the form
\begin{eqnarray} \label{ps}
\mathcal{S}_{h,r,X,Y}\left(\left(\frac{8\cdot}{s}\right),F\right)=\sum_{\substack{d\equiv h\pmod r}}M_Y(d)\left(\frac{8d}{s}\right)F\left(\frac{d}{X}\right)
\end{eqnarray}
where $s$ is an odd natural number, $(s,r)=1$, $X,Y> 1$, 
\begin{eqnarray}\label{MY}
M_Y(d)=\sum_{\substack{k^2|d\\k\leq Y}}\mu(k)\hspace{1cm}
\end{eqnarray}
and $F$ is a real-valued smooth function supported in the interval $(1,2)$. We now prove a version of the Poisson summation formula to cast (\ref{ps}) into a sum involving the Gauss-type sum
\begin{eqnarray}
G_k(s)=\left(\frac{1-i}{2}+\left(\frac{-1}{s}\right)\frac{1+i}{2}\right)\tau_{k}(s)
\end{eqnarray}
where
 \begin{eqnarray}\label{gauss}
\tau_{k}(s)=\sum_{b\pmod s}\left(\frac{b}{s}\right)e\left(\frac{bk}{s}\right)
\end{eqnarray}
and the Fourier transform is
\begin{eqnarray}\label{fouriertdef}
\widehat F(x)=\int_{-\infty}^{\infty}F(y)e(-xy) dy,
\end{eqnarray}
where $e(x)=e^{2\pi i x}$ is the complex exponential.

\begin{lemma}\label{poisson} If $s$ is an odd natural number, $(r,s)=1$,  $X,Y>1$ and $F$ is a real-valued smooth function  supported in the interval $(1,2)$, then 
\begin{eqnarray}
\mathcal{S}_{h,r,X,Y}\left(\left(\frac{8\cdot}{s}\right),F\right)
&=&\frac{X}{rs}\sum_{\substack{\alpha\leq \min(\sqrt{2X},Y)\\(\alpha,s)=1\\(\alpha^2,r)|h}}\frac{(\alpha^2,r)\mu(\alpha)}{\alpha^2}\left(\frac{8r/(\alpha^2,r)}{s}\right)\nonumber\\
&\times& \sum_kG_k(s) \Re\left((1+i)  \widehat F\left(\frac{(\alpha^2,r)kX}{\alpha^2rs}\right)e\left(\frac{\overline{\alpha^2/(\alpha^2,r)}kh\overline s}{r}\right)\right)\nonumber
\end{eqnarray}
where $\overline{s}$ denotes the multiplicative inverse of $s \pmod {r/(\alpha^2,r)}$.
\end{lemma}

\begin{proof}
Note that (\ref{ps}) equals 
\begin{eqnarray}\label{ps2}
&&\left(\frac{8}{s}\right)\sum_{\substack{\alpha\leq \min(\sqrt{2X},Y)}}\mu(\alpha)\left(\frac{\alpha^2}{s}\right)\sum_{\substack{\alpha^2 d\equiv h \pmod r}}\left(\frac{d}{s}\right)F\left(\frac{\alpha^2d}{X}\right)\nonumber\\
&=&\left(\frac{8}{s}\right)\sum_{\substack{\alpha\leq \min(\sqrt{2X},Y)\\(\alpha,s)=1\\(\alpha^2,r)|h}}\mu(\alpha)\sum_{d\equiv H \pmod R}\left(\frac{d}{s}\right)F\left(\frac{\alpha^2d}{X}\right)
\end{eqnarray}
where
\begin{eqnarray}H=\overline{\frac{\alpha^2}{(\alpha^2,r)}}\frac{h}{(\alpha^2,r)}\hspace{0.7cm}\textrm{and}\hspace{0.7cm}R=\frac{r}{(\alpha^2,r)}.\nonumber
\end{eqnarray} 
Because $s$ is odd, the inner summation in (\ref{ps2}) is
\begin{eqnarray}\label{ps3}
\sum_{b\pmod s}\left(\frac{b}{s}\right)\sum_{ds \equiv H-b\pmod R}F\left(\frac{\alpha^2(ds+b)}{X}\right).
\end{eqnarray}
Since $(r,s)=1$, (\ref{ps3}) is 
\begin{eqnarray}\label{ps4}
&&\sum_{b\pmod s}\left(\frac{b}{s}\right)\sum_{d \equiv (H-b)\overline s\pmod R}F\left(\frac{\alpha^2(ds+b)}{X}\right)\nonumber\\
&=&\sum_{b\pmod s}\left(\frac{b}{s}\right)\sum_{d}F\left(\frac{\alpha^2Rs}{X}\left(d+\frac{(H-b)\overline s}{R}+\frac{b}{Rs}\right)\right).
\end{eqnarray}
Applying the classical Poisson summation formula to the inner summation in (\ref{ps4}), one obtains
\begin{eqnarray}\label{poiss}
\frac{X}{\alpha^2Rs}\sum_k\widehat F\left(\frac{kX}{\alpha^2Rs}\right)
e\left(\frac{kH\overline s}{R}\right)\sum_{b\pmod s}\left(\frac{b}{s}\right)e\left(\frac{bk}{Rs}\left(1-s\overline{s}\right)\right).
\end{eqnarray}
Replacing $b$ with $bR$ and using (\ref{gauss}), $G_{-k}(s)=(\frac{-1}{s})G_k(s)$ and $\widehat F(-u)=\overline{\widehat F}(u)$ (because $F$ is real), it follows that (\ref{poiss}) is
\begin{eqnarray}\label{ps5}
&&\frac{X}{\alpha^2Rs}\left(\frac{R}{s}\right)\sum_k\widehat F\left(\frac{kX}{\alpha^2Rs}\right)
e\left(\frac{kH\overline s}{R}\right)\tau_k(s)\nonumber\\
&=&\frac{X}{\alpha^2Rs}\left(\frac{R}{s}\right)\sum_kG_k(s) \Re\left((1+i)  \widehat F\left(\frac{kX}{\alpha^2Rs}\right)e\left(\frac{kH\overline s}{R}\right)\right).
\end{eqnarray}
Substituting (\ref{ps5}) in (\ref{ps2}) we obtain the lemma. \end{proof}

\subsection{The Gauss-type sum}
If $n$ is squarefree then $(\frac{\cdot}{n})$ is a primitive character with conductor $n$ in which case it is easy to see that $G_k(n)=(\frac{k}{n})\sqrt n$. For general odd $n$, Soundararajan (\cite{S2}, Lemma 2.3) has proved the following.

\begin{lemma}\label{gausssum}If $m$ and $n$ are coprime odd integers then $G_k(m)G_k(n)=G_k(mn)$. If $p^{\alpha}$ is the largest power of $p$ dividing $k$ \emph{(}if $k=0$ then $\alpha=\infty$\emph{)} then for $\beta\geq 1$
\begin{eqnarray}\label{hologauss}
  G_k(p^{\beta}) = \begin{cases}
    0 &  \text{if $\beta\leq \alpha$ is odd},
    \\\\
    \phi(p^{\beta})& \text{if $\beta\leq \alpha$ is even} ,
    \\ \\
     -p^{\alpha} & \text{if $\beta= \alpha+1$ is even},
     \\ \\
     \left(\frac{kp^{-\alpha}}{p}\right)p^{\alpha}\sqrt p & \text{if $\beta= \alpha+1$ is odd},
     \\\\
     0 & \text{if $\beta\geq\alpha+2$.}
    \end{cases}\nonumber\\
\end{eqnarray}
\end{lemma}

\subsection{Approximate functional equations}\label{apeqs}
In this section we recall the `approximate' functional equations that we require for $L\left(\frac{1}{2},\psi\otimes\chi_{8d}\right)$ and $|L\left(\frac{1}{2},\psi\otimes\chi_{8d}\right)|^2$. Firstly, we need the following definition.
\begin{definition}\label{def1} For $\eta,\xi>0$ and $j=1$ or $j=2$,
\begin{eqnarray}
\omega_j(\xi)=\frac{1}{2\pi i}\int_{(c>0)}\left(\frac{\Gamma\left(\frac{s}{2}+\frac{1}{4}\right)}{\Gamma\left(\frac{1}{4}\right)}\right)^j\xi^{-s}\frac{ds}{s}
\end{eqnarray}
and 
\begin{eqnarray}\label{Fdef}
F_{j,\eta}(\xi)=\Phi(\xi)\omega_j\left(\eta\left(\frac{\pi}{8qX\xi}\right)^{j/2}\right).
\end{eqnarray}
\end{definition}
\noindent Then Soundararajan has proved the following. 
\begin{lemma}\label{xibounds}For any positive integer $\nu$ we have
\begin{eqnarray}
\omega_j^{(\nu)}(\xi)\ll_{\nu,j}\xi^{2\nu+2}\exp\left(-j\xi^{2/j}\right)\ll_{\nu,j} \exp\left(-\frac{j\xi^{2/j}}{2}\right)
\end{eqnarray}
as $\xi\rightarrow\infty$. Moreover, for any fixed $\nu>0$ there is a $C>0$ such that
\begin{eqnarray}\label{fbound}
\widehat F_{j,\eta}(\xi)\ll_{\nu,\Phi}\exp\left(-\frac{C\eta^{2/j}}{X}\right)\left|\xi\right|^{-\nu}
\end{eqnarray}
as $\eta\rightarrow\infty$ or $\xi\rightarrow\infty$.
 \end{lemma}

We also denote by $\tau(\chi)$ the Gauss sum 
\begin{eqnarray}
\tau(\chi)=\sum_{a\pmod N}\chi(a)e\left(\frac{a}{N}\right).
\end{eqnarray}
The approximate functional equations that we require are as follows.
\begin{lemma}\label{dlem}
Let $\psi=1$ or $\psi$ be an even primitive non-quadratic character of odd modulus $q$ and let $\chi_{8d}$ be a primitive quadratic character. Also  let $h$ be odd, $r$ even, $q|r$, $d\equiv h\pmod r$, $(h,r)=1$, 
\begin{eqnarray}\mathcal{S}(r)=\{m\in\mathbb{N}:p|m\Rightarrow p|r\},
\end{eqnarray} 
\begin{eqnarray}
\epsilon(d)=\frac{\tau(\psi)}{q^{1/2}}\psi(8d)\chi_{8d}(q)=\epsilon(h),
\end{eqnarray}
and 
\begin{eqnarray}\label{dddef}
d_{\psi}(n)=\sum_{d|n}\psi(d)\overline\psi\left(\frac{n}{d}\right).
\end{eqnarray}
Then we have
\begin{eqnarray}\label{appeq5}
L\left(\frac{1}{2},\psi\otimes\chi_{8d}\right)&=&\sum_{m\in\mathcal{S}(r)}\sum_{(n,r)=1}\frac{\chi_{8h}(m)\chi_{8d}(n)}{(mn)^{1/2}}\left(\psi(m)\psi(n)
+\epsilon(h)\overline \psi(m)\overline \psi(n)\right)\nonumber\\
&\times&\omega_1\left(mn\sqrt\frac{\pi}{8dq}\right)
\end{eqnarray}
and
\begin{eqnarray} \label{appeq6}
\left|L\left(\frac{1}{2},\psi\otimes\chi_{8d}\right)\right|^2&=&2\sum_{m\in \mathcal{S}(r)}\frac{d_{\psi}(m) \chi_{8h}(m)}{m^{1/2}}\sum_{(n,r)=1}\frac{d_{\psi}(n)\chi_{8d}(n)}{n^{1/2}}\omega_2\left(\frac{mn\pi}{8dq}\right).\nonumber\\
\end{eqnarray}
\end{lemma}

\begin{proof}For an even primitive Dirichlet character $\chi\pmod N$, the completed Dirichlet $L$-function 
\begin{eqnarray}\label{def2}
\xi(s,\chi)=\left(\frac{\pi}{N}\right)^{s/2}\Gamma\left(\frac{s}{2}\right)L(s,\chi)
\end{eqnarray}
is an entire function and satisfies the functional equation 
\begin{eqnarray} \label{func}
\xi(1-s,\overline{\chi})=\frac{N^{1/2}}{\tau(\chi)}\xi(s,\chi).
\end{eqnarray}
Due to the rapid decay of the $\Gamma$ function and (\ref{func}), for $c>0$ we have 
\begin{eqnarray} \label{int1}
\xi\left(\frac{1}{2},\chi\right)=\frac{1}{2\pi i}\int_{(c)}\left(\xi\left(s+\frac{1}{2},\chi\right)+\frac{\tau(\chi)}{N^{1/2}}\xi\left(s+\frac{1}{2},\overline{\chi}\right)\right)\frac{ds}{s}
\end{eqnarray}
and 
\begin{eqnarray} \label{int2}
\left|\xi\left(\frac{1}{2},\chi\right)\right|^2=\frac{1}{\pi i}\int_{(c)}\xi\left(s+\frac{1}{2},\chi\right)\xi\left(s+\frac{1}{2},\overline{\chi}\right)\frac{ds}{s}.
\end{eqnarray}
By  (\ref{func}), for $c>1/2$ it follows that (\ref{int1}) and (\ref{int2}) are
\begin{eqnarray} \label{eq1}
L\left(\frac{1}{2},\chi\right)=\sum_{n=1}^{\infty}\frac{1}{n^{1/2}}\left(\chi(n)+\frac{\tau(\chi)}{N^{1/2}}\overline\chi(n)\right)
\frac{1}{2\pi i}\int_{(c)}\frac{\Gamma\left(\frac{s}{2}+\frac{1}{4}\right)}{\Gamma\left(\frac{1}{4}\right)}\left(n\sqrt\frac{\pi}{N}\right)^{-s}\frac{ds}{s}\nonumber\\
\end{eqnarray}
and 
\begin{eqnarray} \label{eq2}
\left|L\left(\frac{1}{2},\chi\right)\right|^2=\sum_{n_1,n_2=1}^{\infty}\frac{\chi(n_1)\overline\chi(n_2)}{(n_1n_2)^{1/2}}\frac{1}{\pi i}\int_{(c)}\left(\frac{\Gamma\left(\frac{s}{2}+\frac{1}{4}\right)}{\Gamma\left(\frac{1}{4}\right)}\right)^2\left(\frac{n_1n_2\pi}{N}\right)^{-s}\frac{ds}{s}.\nonumber\\
\end{eqnarray}
Using Definition \ref{def1} in (\ref{eq1}) and (\ref{eq2}), we obtain
\begin{eqnarray} \label{appeq1}
L\left(\frac{1}{2},\chi\right)&=&\sum_{n=1}^{\infty}\frac{1}{n^{1/2}}\left(\chi(n)+\frac{\tau(\chi)}{N^{1/2}}\overline\chi(n)\right)
\omega_1\left(n\sqrt\frac{\pi}{N}\right)
\end{eqnarray}
and
\begin{eqnarray} \label{appeq2}
\left|L\left(\frac{1}{2},\chi\right)\right|^2=2\sum_{n=1}^{\infty}\frac{d_{\chi}(n)}{n^{1/2}}\omega_2\left(\frac{n_1n_2\pi}{N}\right).
\end{eqnarray}
Now it is elementary that $\tau(\chi_1\chi_2)=\tau(\chi_1)\tau(\chi_2)\chi_1(N_2)\chi_2(N_1)$ for primitive characters modulo $N_1$, $N_2$ if $(N_1,N_2)=1$. Thus, taking $N_1=q$, $N_2=8d$ and $\chi=\psi\chi_{8d}$ in (\ref{appeq1}) and (\ref{appeq2}), we get 
\begin{eqnarray} \label{appeq3}
L\left(\frac{1}{2},\psi\otimes\chi_{8d}\right)=\sum_{n=1}^{\infty}\frac{\chi_{8d}(n)}{n^{1/2}}\left(\psi(n)+\epsilon(h)\overline\psi(n)\right)
\omega_1\left(n\sqrt\frac{\pi}{8dq}\right)
\end{eqnarray}
and
\begin{eqnarray} \label{appeq4}
\left|L\left(\frac{1}{2},\psi\otimes\chi_{8d}\right)\right|^2=2\sum_{n=1}^{\infty}\frac{d_{\psi}(n)\chi_{8d}(n)}{n^{1/2}}\omega_2\left(\frac{n\pi}{8dq}\right).
\end{eqnarray}
Since $q|r$, for any $m\in\mathcal{S}(r)$ we have $\chi_{8d}(m)=\chi_{8h}(m)$. Thus (\ref{appeq5}) and (\ref{appeq6}) follow from (\ref{appeq3}) and (\ref{appeq4}).
\end{proof}

\section{Proof of Theorem \ref{th1}}\label{ev}
Following Soundarajan \cite{S2}, we evaluate the expected value
\begin{eqnarray}\label{mom3d}
\mathcal{E}=\frac{1}{X}\sum_{\substack{d\equiv h \pmod r}}\mu^2(d)\chi_{8d}\left(l\right)L(1/2,\psi\otimes\chi_{8d})\Phi\left(\frac{d}{X}\right)
\end{eqnarray}
by writing
\begin{eqnarray}
\mu^2(d)=M_Y(d)+R_Y(d)\nonumber
\end{eqnarray} 
where $M_Y(d)$ is defined in (\ref{MY}),
\begin{eqnarray}
R_Y(d)=\sum_{\substack{k^2|d\\k> Y}}\mu(k)
\end{eqnarray}
and clearly $R_Y(d)\ll d^{\epsilon}$. Then
\begin{eqnarray}\label{Rbound1}
&&\mathcal{E}-\frac{1}{X}\sum_{\substack{d\equiv h \pmod r}}M_Y(d)\chi_{8d}\left(l\right)L(1/2,\psi\otimes\chi_{8d})\Phi\left(\frac{d}{X}\right)\nonumber\\
&\leq&\frac{1}{X} \sum_d \left|R_Y(d)L(1/2,\psi\otimes\chi_{8d})\Phi\left(\frac{d}{X}\right)\right|\nonumber\\
&\leq&X^{\epsilon/2}\left(\frac{1}{X} \sum_d \left|R_Y(d)L^2(1/2,\psi\otimes\chi_{8d})\Phi\left(\frac{d}{X}\right)\right|\right)^{1/2}
\end{eqnarray}
by the Cauchy-Schwartz inequality, and Soundararajan (\cite{S1}, Lemma 3.2) has shown that the right hand side of (\ref{Rbound1}) is
\begin{eqnarray}
\ll \frac{X^{\epsilon}}{Y^{1/2}}.
\end{eqnarray}
Inserting the approximate functional equation for $L(1/2,\psi\otimes\chi_{8d})$ of Lemma \ref{dlem} on the left hand side of (\ref{Rbound1}), interchanging the order of summation we obtain
\begin{eqnarray}\label{Msum}
\mathcal{E}&=&\frac{1}{X}\sum_{m\in\mathcal{S}(r)}\sum_{(n,r)=1}\frac{\chi_{8h}(m)\chi_{8d}(n)}{(mn)^{1/2}}\left(\psi(m)\psi(n)
+\epsilon(h)\overline \psi(m)\overline \psi(n)\right)
\nonumber\\&\times&\mathcal{S}_{h,r,X,Y}\left(\left(\frac{8\cdot}{ln}\right),F_{1,mn}\right)+O\left(\frac{X^{\epsilon}}{Y^{1/2}}\right)
\end{eqnarray}
and applying the Poisson summation formula of Lemma \ref{pois} in (\ref{Msum}), we get 
\begin{eqnarray}\label{evdecomp}
\mathcal{E}=\mathcal{D}+\mathcal{N}+O\left(\frac{X^{\epsilon}}{Y^{1/2}}\right)
\end{eqnarray}
where
\begin{eqnarray}\label{evd}
\mathcal{D}&=&\frac{1}{r}\sum_{m\in\mathcal{S}(r)}\sum_{(n,r)=1}\frac{\chi_{8h}(m)}{(mn)^{1/2}}\frac{\widehat F_{1,mn}(0)G_0(ln)}{ln}\left(\frac{8r}{ln}\right)\nonumber\\
&\times&\left(\psi(m)\psi(n)
+\epsilon(h)\overline \psi(m)\overline \psi(n)\right)\sum_{\substack{\alpha\leq Y\\(\alpha,lnr)=1\\}}\frac{\mu(\alpha)}{\alpha^2}
\end{eqnarray}
denotes the diagonal term and 
\begin{eqnarray}\label{evndiag}
\mathcal{N}&=&\frac{1}{r}\sum_{m\in\mathcal{S}(r)}\sum_{(n,r)=1}\frac{\chi_{8h}(m)}{(mn)^{1/2}}\left(\frac{8r}{ln}\right)\left(\psi(m)\psi(n)
+\epsilon(h)\overline \psi(m)\overline \psi(n)\right)\nonumber\\
&\times&\sum_{\substack{\alpha\leq Y\\(\alpha,lnr)=1\\}}\frac{\mu(\alpha)}{\alpha^2}\sum_{k\neq 0}\frac{G_k(ln)}{ln} \Re\left((1+i)  \widehat F_{1,mn}\left(\frac{kX}{\alpha^2lnr}\right)e\left(\frac{\overline{\alpha^2ln}kh}{r}\right)\right)
\end{eqnarray}
denotes the non-diagonal terms. Clearly if $(l,r)>1$ then $\mathcal{D}=\mathcal{N}=0$, so we assume that $(l,r)=1$ in what follows.

\subsection{Bounds on the non-diagonal terms}\label{bn} 
Bounding everything trivially (using $G_k(s)\ll s$), the non-diagonal terms (\ref{evndiag}) are 
\begin{eqnarray}\label{ndiagboundo}
\mathcal{N}\ll\frac{1}{r}\sum_{m\in\mathcal{S}(r)}\sum_{(n,r)=1}\sum_{\substack{\alpha\leq Y\\}}\frac{1}{\alpha^2(mn)^{1/2}}\sum_{k\neq 0}\left|\widehat F_{1,mn}\left(\frac{kX}{\alpha^2lnr}\right)\right|.
\end{eqnarray}
By Lemma \ref{xibounds}, the terms with $n\gg X^{1/2+\epsilon}$ in (\ref{ndiagboundo}) are exponentially decreasing in $X$. For the remaining $n$, we have $|kX/\alpha^2lnr|\geq X^{\epsilon}$ provided that $lrY^2\leq X^{1/2-2\epsilon}$, so (\ref{ndiagboundo}) is
\begin{eqnarray}\label{boundq}
&\ll&\frac{r^{\nu-1}}{X^{\nu}}\sum_{m\leq r}\sum_{n\leq X^{1/2+\epsilon}}\sum_{\substack{\alpha\leq Y\\}}\sum_{k\neq 0}\frac{1}{\alpha^{2-2\nu} k^{\nu}m^{1/2}n^{1/2-\nu}}\nonumber\\
&\ll&\frac{l^{\nu}r^{\nu-1/2}Y^{2\nu-1}}{X^{\nu/2-1/4-\epsilon}}\nonumber
\end{eqnarray}
for any $\nu>1$. Taking $\nu=3/2$, we get
\begin{eqnarray}\label{ndiagbound}
\mathcal{N} \ll \frac{l^{3/2}r Y^2}{X^{1/2-\epsilon}}.
\end{eqnarray}

\subsection{The diagonal term}\label{diagevm}
Since $G_0(s)=0$ if $s\neq \square$ and $G_0(s)=\phi(s)$ if $s=\square$,  we put $l=l_1l_2^2$ where $l_1$ is squarefree so that $ln=\square$ is equivalent to having $n\rightarrow l_1n^2$. Thus, the diagonal term (\ref{evd}) is 
\begin{eqnarray}\label{diagev2}
\mathcal{D}&=&\frac{1}{l_1^{1/2}r}\sum_{m\in\mathcal{S}(r)}\frac{\chi_{8h}(m)}{m^{1/2}}\sum_{(l_1n,r)=1}\frac{\hat F_{1,l_1mn^2}(0)\prod_{p|ln}(1-p^{-1})}{n}\nonumber\\
&\times&\left(\psi(l_1) \psi(m) \psi^2(n)+\epsilon(h)\overline\psi(l_1) \overline \psi(m)  \overline  \psi^2(n)  \right)
\sum_{\substack{\alpha\leq Y\\(\alpha,lnr)=1\\}}\frac{\mu(\alpha)}{\alpha^2}.
\end{eqnarray}
Since $\hat F_{1,l_1mn^2}(0)$ is exponentially decreasing in $X$ when $n\gg X^{1/4+\epsilon}$, completing the sum over $\alpha$ in (\ref{diagev2}) introduces an error of
\begin{eqnarray}\label{diagev3}
&\ll&\frac{1}{l_1^{1/2}rY}\sum_{m\in\mathcal{S}(r)}\frac{1}{m^{1/2}}\sum_{n\leq X^{1/4+\epsilon}}\frac{1}{n}\nonumber\\
&\ll&\frac{\log X}{(l_1r)^{1/2}Y}.
\end{eqnarray}
Since also
\begin{eqnarray}
\prod_{p|ln}(1-p^{-1})\sum_{(\alpha,lnr)=1}\frac{\mu(\alpha)}{\alpha^2}&=&\prod_{p|ln}(1-p^{-1})\prod_{p\nmid lnr}\left(1-p^{-2}\right)
\nonumber\\&=&\frac{1}{L(2,\varphi_0)\prod_{p|ln}\left(1+p^{-1}\right)}\nonumber
\end{eqnarray}
where $\varphi_0$ is the principal character to the modulus $r$, the diagonal term (\ref{evd}) is 
\begin{eqnarray}\label{gcd}
\mathcal{D}=D_{h,l,r,X,\psi}+\epsilon(h)\overline{D_{h,l,r,X,\psi}}+O\left(\frac{\log X}{(l_1r)^{1/2}Y}\right)
\end{eqnarray}
where 
\begin{eqnarray}\label{an0}
D_{h,l,r,X,\psi}=\frac{\varphi_0(l_1)\psi(l_1)}{L(2,\varphi_0)l_1^{1/2}r}    \sum_{m\in\mathcal{S}(r)}\left(\frac{8h}{m}\right)\frac{\psi(m)}{m^{1/2}}
\sum_{1}^{\infty}\frac{\varphi_0(n)\psi^2(n)\widehat F_{1,l_1mn^2}(0)}{n\prod_{p| ln}\left(1+p^{-1}\right)}.\nonumber\\
\end{eqnarray}\\

To evaluate (\ref{an0}), for Dirichlet characters $\chi$ we define the following Euler products
\begin{eqnarray}\label{adef}
A_{\nu}(s,\chi)=\prod_{p|\nu}\left(1-\frac{\chi(p)}{p^s}\right),\nonumber
\end{eqnarray}
\begin{eqnarray}
B_{\nu}(s,\chi)=\prod_{p\nmid \nu}\left(1-\frac{\chi(p)}{p^{s}(p+1)}\right),\nonumber
\end{eqnarray}
\begin{eqnarray}
C_{\mu,\nu}(s,\chi)=\prod_{\substack{p\nmid \mu\\p|\nu}}\left(1+\frac{1}{p}\right)\left(1-\frac{\chi(p)}{p^{s}(p+1)}\right)\nonumber
\end{eqnarray}
and 
 \begin{eqnarray}\label{phimel}
\check\Phi\left(s\right)=\int_{0}^{\infty}\Phi(y)y^sdy\hspace{1.5cm}(s\in\mathbb{C}).
\end{eqnarray}
Using Definition \ref{def1} to write $\widehat F_{1,l_1mn^2}(0)$ as a double integral and using (\ref{phimel}), we get
\begin{eqnarray}\label{ant1}
D_{h,l,r,X,\psi}=\frac{\varphi_0(l_1)\psi(l_1)}{L(2,\varphi_0)l_1^{1/2}r} \frac{1}{2\pi i}    \int_{(c>0)}\left(\sum_{1}^{\infty}\frac{\varphi_0(n)\psi^2(n)}{n^{2s+1}\prod_{p| ln}\left(1+p^{-1}\right)}\right)\frac{\Gamma\left(\frac{s}{2}+\frac{1}{4}\right)\check\Phi\left(\frac{s}{2}\right)\left(\frac{8qX}{\pi}\right)^{s/2}ds}{A_{r/q}\left(s+\frac{1}{2},\chi_{8h}\psi\right)\Gamma\left(\frac{1}{4}\right)l_1^{s}s},
\nonumber\\
\end{eqnarray}
where 
\begin{eqnarray}
\sum_{1}^{\infty}\frac{\varphi_0(n)\psi^2(n)}{n^{2s+1}\prod_{p| ln}\left(1+p^{-1}\right)}=\frac{A_{r}(2s+1,\psi^2)B_r(2s+1,\psi^2)L(2s+1,\psi^2)}{C_{r,l}(2s+1,\psi^2)}.\nonumber
\end{eqnarray}
Since $B(s,\chi)$ is bounded for fixed $\sigma>0$ and $L(s,\chi)$ is polynomially bounded for $\sigma>0$ except for a pole at $s=1$ when $\chi$ is principal,
the rapid decay of the $\Gamma$ function and $\tilde\Phi$ permit us to move the path of integration in \eqref{ant1} to the line $\sigma=\epsilon-1/2$ ($0<\epsilon<1/2$) so that the integral in \eqref{ant1} is equal to 

\begin{eqnarray}\label{res}
&&\textrm{Res}\left(\frac{A_{r}(2s+1,\psi^2)B_r(2s+1,\psi^2)L(2s+1,\psi^2)\Gamma\left(\frac{s}{2}+\frac{1}{4}\right)\check\Phi\left(\frac{s}{2}\right)\left(\frac{8qX}{\pi}\right)^{s/2}}{A_{r/q}\left(s+\frac{1}{2},\chi_{8h}\psi\right)C_{r,l}(2s+1,\psi^2)\Gamma\left(\frac{1}{4}\right)l_1^{s}s},s=0\right)\nonumber\\
&+&\frac{1}{2\pi i}\int_{(\epsilon-1/2)}
\frac{A_{r}(2s+1,\psi^2)B_r(2s+1,\psi^2)L(2s+1,\psi^2)\Gamma\left(\frac{s}{2}+\frac{1}{4}\right)\check\Phi\left(\frac{s}{2}\right)\left(\frac{8qX}{\pi}\right)^{s/2}ds}
{A_{r/q}\left(s+\frac{1}{2},\chi_{8h}\psi\right)C_{r,l}(2s+1,\psi^2)\Gamma\left(\frac{1}{4}\right)l_1^{s}s}.
  \nonumber\\
\end{eqnarray}
Since clearly $A_{\nu}(\epsilon,\chi)$, $C^{-1}_{\mu,\nu}(\epsilon,\chi)$ and $A^{-1}_{\nu}(\epsilon,\chi)$  are all $\ll_{\epsilon}\nu^{\epsilon}$, the integral in (\ref{res}) is 
\begin{eqnarray}\label{intbound}
\ll_{\epsilon}l^{\epsilon}l_1^{1/2-\epsilon}r^{\epsilon-1}X^{\epsilon-1/4}
\end{eqnarray}
and, combining the error terms (\ref{evdecomp}), (\ref{ndiagbound}), (\ref{diagev3}) and (\ref{intbound}), we obtain the error term (\ref{th1errorbound}) of Theorem \ref{th1}. To get the main terms, we evaluate the residue in (\ref{res}). If $\psi$ is non-quadratic then $L(2s+1,\psi^2)$ is analytic at $s=0$ so the pole is simple and the residue is
\begin{eqnarray}\label{simpole}
\frac{\widehat\Phi\left(0\right)A_{r}(1,\psi^2)B_r(1,\psi^2)L(1,\psi^2)}{A_{r/q}\left(\frac{1}{2},\chi_{8h}\psi\right)C_{r,l}(1,\psi^2)},
\end{eqnarray}
which proves (\ref{mom3}). On the other hand, if $\psi=1$ is the trivial character, then $L(s,\psi^2)=\zeta(s)$ and 
\begin{eqnarray}
\frac{\zeta(2s+1)}{s}=\frac{1}{2s^2}+\frac{\gamma}{s}+O(1)\nonumber
\end{eqnarray}
as $s\rightarrow\infty$, so the residue is 
\begin{eqnarray}\label{dipole}
&=&\frac{\widehat\Phi\left(0\right)A_{r}(1,1)B_r(1,1)L(1,1)}{A_{r/q}\left(\frac{1}{2},\chi_{8h}\right)C_{r,l}(1,1)}\left(\frac{\log X}{4}+\gamma\right)
\nonumber\\
&+&\frac{1}{2}\frac{d}{ds}
\frac{A_{r}(2s+1,1)B_r(2s+1,1)\Gamma\left(\frac{s}{2}+\frac{1}{4}\right)\check\Phi\left(\frac{s}{2}\right)\left(\frac{8}{\pi}\right)^{s/2}}
{A_{r/q}\left(s+\frac{1}{2},\chi_{8h}\psi\right)C_{r,l}(2s+1,1)\Gamma\left(\frac{1}{4}\right)l_1^{s}}\biggr\vert _{s=0},\nonumber
\end{eqnarray}
which completes our proof of Theorem \ref{th1}.

\section{Proof of Theorem \ref{th2}}\label{sm}
To handle the second moments
\begin{eqnarray}\label{mom4d}
\mathcal{M}_2=\frac{1}{X}\sum_{\substack{d\equiv h \pmod r}}\mu^2(d)\chi_{8d}\left(l\right)\left|L(1/2,\psi\otimes\chi_{8d})\right|^2\Phi\left(\frac{d}{X}\right),
\end{eqnarray} 
we begin as we did in the proof of Theorem \ref{th1} in Section \ref{ev}. Here the analogue of (\ref{Rbound1}) is
\begin{eqnarray}\label{Rbound2}
\mathcal{M}_2-\frac{1}{X}\sum_{\substack{d\equiv h \pmod r}}M_Y(d)\chi_{8d}\left(l\right)\left|L(1/2,\psi\otimes\chi_{8d})\right|^2\Phi\left(\frac{d}{X}\right)\ll \frac{X^{\epsilon}}{Y}.
\end{eqnarray}\\
Using the approximate functional equation for $|L(1/2,\psi\otimes\chi_{8d})|^2$ from Lemma \ref{dlem} in (\ref{Rbound2}) and interchanging the order of summation, we have  
\begin{eqnarray}\label{Msum2}
\mathcal{M}_2=\frac{2}{X}\sum_{m\in\mathcal{S}(r)}\frac{\chi_{8h}(m)d_\psi(m)}{m^{1/2}}\sum_{(n,r)=1}\frac{d_\psi(n)}{n^{1/2}}\mathcal{S}_{h,r,X,Y}\left(\left(\frac{8\cdot}{ln}\right),F_{2,mn}\right)+O\left(\frac{X^{\epsilon}}{Y}\right).\nonumber\\
\end{eqnarray}
Applying the Poisson summation formula (Lemma \ref{poisson}) in (\ref{Msum2}), we get
\begin{eqnarray}\label{decomp}
\mathcal{M}_2=\mathcal{P}_1+\mathcal{R}_0+O\left(\frac{X^{\epsilon}}{Y}\right)
\end{eqnarray}
where
\begin{eqnarray}\label{both2}
\mathcal{P}_1=\frac{2}{rl}\sum_{m\in\mathcal{S}(r)}\frac{\chi_{8h}(m)d_\psi(m)}{m^{1/2}}\sum_{(n,r)=1}\frac{d_\psi(n)}{n^{1/2}} \frac{\hat F_{2,mn}(0)G_0(ln)}{n}\left(\frac{8r}{ln}\right)\sum_{\substack{\alpha\leq Y\\(\alpha,lnr)=1\\}}\frac{\mu(\alpha)}{\alpha^2}\nonumber\\
\end{eqnarray}
denotes the diagonal term, and 
\begin{eqnarray}\label{secmndiag}
\mathcal{R}_0&=&\frac{2}{lr}\sum_{m\in\mathcal{S}(r)}\frac{\chi_{8h}(m)d_\psi(m)}{m^{1/2}}\sum_{(n,r)=1}\frac{d_\psi(n)}{n^{1/2}}\sum_{\substack{\alpha\leq Y\\(\alpha,lnr)=1\\}}\frac{\mu(\alpha)}{\alpha^2}\left(\frac{r/2}{ln}\right)\nonumber\\
&\times& \sum_{k\neq 0}\frac{G_{4k}(ln)}{n} \Re\left((1+i)  \widehat F_{2,mn}\left(\frac{kX}{\alpha^2lnr}\right)e\left(\frac{\overline{\alpha^2ln}kh}{r}\right)\right)
\end{eqnarray}
denotes the non-diagonal terms. Clearly if $(l,r)>1$ then $\mathcal{P}_1=\mathcal{R}_0=0$, so we assume that $(l,r)=1$ in what follows.  We evaluate $\mathcal{P}_1$ in  Section \ref{diagsm} and $\mathcal{R}_0$ in Section \ref{dc}.

\subsection{The non-diagonal terms}\label{dc}
The objective of this section is to show that $\mathcal{R}_0=\mathcal{P}_2+\mathcal{R}$ where $\mathcal{P}_2$ is a non-diagonal main term and $\mathcal{R}$ is an error term. We use the notation 
$\textrm{cas}(x)=\cos( x)+\sin( x)$ and follow Soundararajan \cite{S2} in defining
\begin{eqnarray}
\widetilde F(x)=\Re \widehat F (x)-\Im \widehat F (x)=\int_{-\infty}^{\infty}F(y)\textrm{cas}(2\pi xy)dy
\end{eqnarray}
for real $F$. As such,
\begin{eqnarray}\label{cas}
\Re\left((1+i) \widehat F\left(x\right)e\left(z\right) \right)=\cos(2\pi z)\widetilde F(x)+\sin(2\pi z)\widetilde F(-x).
\end{eqnarray}
Assuming that $lrY^2\ll X^{1/2-\delta}$ for some fixed $\delta>0$ and interchanging the order of the summations over $k$ and $n$ in (\ref{secmndiag}) (this is justified by Lemma \ref{xibounds}), we are lead to consider sums of the form
\begin{eqnarray}\label{nsum}
&&\sum_{\substack{(n,\alpha r)=1}}\left(\frac{r/2}{ln}\right)\frac{d_{\psi}(n)}{n}\frac{G_{4k}(ln)}{n^{1/2}}\nonumber\\
&\times&\left(\cos\left(\frac{2\pi kh\overline{\alpha^2ln}}{r}\right)\widetilde F_{2,mn}\left(\frac{kX}{\alpha^2lnr}\right)+\sin\left(\frac{2\pi kh\overline{\alpha^2ln}}{r}\right)\widetilde F_{2,mn}\left(-\frac{kX}{\alpha^2lnr}\right)\right).\nonumber\\
\end{eqnarray} 
In particular, the terms with $n\gg X^{1+\epsilon}$ in (\ref{nsum}) are exponentially decreasing in $X$ so, by Mellin inversion, the sum (\ref{nsum}) may be written as the integral  
\begin{eqnarray}\label{melinv}
\frac{1}{2\pi i}\int_{(c>1/2)} \left(    \mathcal{C}_{\psi}(s+1)    f\left(\frac{kmX}{\alpha^2lr},s\right)+\mathcal{S}_{\psi}(s+1)  f\left(-\frac{kmX}{\alpha^2lr},s\right)      \right)      \frac{ds}{m^s}\nonumber\\
\end{eqnarray}
where
\begin{eqnarray}\label{dir}
\mathcal{C}_{\psi}(s)=\sum_{\substack{(n,\alpha r)=1}}\left(\frac{r/2}{ln}\right)\frac{d_{\psi}(n)}{n^{s}}\frac{G_{4k}(ln)}{n^{1/2}}\cos\left(\frac{2\pi kh\overline{\alpha^2ln}}{r}\right)
\end{eqnarray}
and 
\begin{eqnarray}\label{dirsin}
\mathcal{S}_{\psi}(s)=\sum_{\substack{(n,\alpha r)=1}}\left(\frac{r/2}{ln}\right)\frac{d_{\psi}(n)}{n^{s}}\frac{G_{4k}(ln)}{n^{1/2}}\sin\left(\frac{2\pi kh\overline{\alpha^2ln}}{r}\right)
\end{eqnarray}
converge absolutely for $\sigma>3/2$ and
\begin{eqnarray}
 f\left(\xi,s\right)=\int_{0}^{\infty} \widetilde F_{2,t}\left(\frac{\xi}{t}\right)t^{s-1}dt
\end{eqnarray}
converges absolutely for $\sigma>0$.

\subsubsection{The non-diagonal main term}\label{exp} To understand the functions defined by the Dirichlet series (\ref{dir}) and (\ref{dirsin}), we expand the $r$-periodic trigonometric functions as linear combinations of Dirichlet characters (mod $r$), e.g. 
\begin{eqnarray}
\cos\left(\frac{2\pi kh\overline{x}}{r}\right)=\frac{1}{\phi(r)}\sum_{\varphi\pmod r}\left\langle \cos\left(\frac{2\pi k\cdot}{r}\right) ,\overline\varphi  \right\rangle \varphi\left(\overline{h}x\right)
\end{eqnarray}
where 
\begin{eqnarray}
\langle  f,g\rangle=\sum_{a\pmod r}f(a) \overline{g(a)}
\end{eqnarray}
 is the inner product on the group $\mathbb{Z}/r\mathbb{Z}$, so that (\ref{dir}) is equal to 
\begin{eqnarray}\label{expansion}
\mathcal{C}_{\psi}(s)=\frac{1}{\phi(r)}\sum_{\varphi\pmod r}\left\langle \cos\left(\frac{2\pi k\cdot}{r}\right) ,\overline\varphi  \right\rangle \varphi\left(\overline h\alpha^2l\right)\mathcal{D}_{\psi,\varphi}(s;k,l,r,\alpha)\nonumber\\
\end{eqnarray}
where  $\mathcal{D}_{\psi,\varphi}(s;k,l,r,\alpha)$ is the Dirichlet series/Euler product
\begin{eqnarray}\label{euler}
\mathcal{D}_{\psi,\varphi}(s;k,l,r,\alpha)&=&\sum_{\substack{(n,\alpha r)=1}}\left(\frac{r/2}{n}\right)\frac{d_{\psi}(n)\varphi(n)}{n^{s}}\frac{G_{4k}(ln)}{n^{1/2}}\nonumber\\
&=&\prod_{p\nmid \alpha r}\sum_{\beta=0}^{\infty}\left(\frac{r/2}{p^{\beta}}\right)\frac{d_{\psi}(p^{\beta})\varphi(p^{\beta})}{p^{\beta s}}\frac{G_{4k}(p^{\beta+\delta})}{p^{\beta/2}}
\end{eqnarray}
in which $l=\prod p^{\delta}$. Now let $4k=k_1k^2_2$ with $k_1$ a fundamental discriminant. A typical prime $p\nmid \alpha r$ neither divides $k$ or $l$ so, by Lemma \ref{gausssum}, a typical Euler factor in (\ref{euler}) is 
\begin{eqnarray}\label{factors}
&=&1+\left(\frac{r/2}{p}\right)\frac{d_{\psi}(p)\varphi(p)}{p^s}\left(\frac{4k}{p}\right)\nonumber\\ &=& 1+\left(\frac{k_1r/2}{p}\right)\frac{d_{\psi}(p)\varphi(p)}{p^s}\nonumber\\&=&\left(1-\left(\frac{k_1r/2}{p}\right)\frac{\psi(p)\varphi(p)}{p^s}\right)^{-1}\left(1-\left(\frac{k_1r/2}{p}\right)\frac{\overline\psi(p)\varphi(p)}{p^s}\right)^{-1}+O(p^{-2\sigma}).\nonumber
\end{eqnarray}
As such, we write 
\begin{eqnarray}\label{maineuler}
\mathcal{G}_{\psi,\varphi}(s;k,l,r,\alpha)=\prod_{p} \mathcal{G}_{\psi,\varphi,p}(s;k)
\end{eqnarray}
where
\begin{eqnarray}\label{holo}
    \mathcal{G}_{\psi,\varphi,p}(s;k)=
    \begin{cases}
      \left(1-\left(\frac{k_1r/2}{p}\right)\frac{\psi(p)\varphi(p)}{p^s}\right)\left(1-\left(\frac{k_1r/2}{p}\right)\frac{\overline\psi(p)\varphi(p)}{p^s}\right) & \text{}p|\alpha r\\
      \left(1-\left(\frac{k_1r/2}{p}\right)\frac{\psi(p)\varphi(p)}{p^s}\right)\left(1-\left(\frac{k_1r/2}{p}\right)\frac{\overline\psi(p)\varphi(p)}{p^s}\right)\\
      \times\sum_{\beta=0}^{\infty}\left(\frac{r/2}{p^{\beta}}\right)\frac{d_{\psi}(p^{\beta})\varphi(p^{\beta})}{p^{\beta s}}\frac{G_{4k}(p^{\beta+\delta})}{p^{\beta/2}} & \text{}p\nmid \alpha r,
    \end{cases}\nonumber\\
\end{eqnarray}
so that 
\begin{eqnarray}\label{denom}
\mathcal{G}_{\psi,\varphi}(s;k,l,r,\alpha)=\frac{\mathcal{D}_{\psi,\varphi}(s;k,l,r,\alpha)}{L\left(s,   \psi\varphi\left(\frac{k_1r/2}{\cdot}   \right)\right)L\left(s,\overline\psi\varphi\left(\frac{k_1r/2}{\cdot}   \right)\right)}
\end{eqnarray}
is holomorphic in the half plane $\sigma>1/2$ and bounded in $\sigma>1/2+\epsilon$ for any fixed $k$, $l$, $\alpha$ and $\epsilon>0$. Specifically, we have the bound
\begin{eqnarray}\label{Gbound}
\left|\mathcal{G}_{\psi,\varphi}(s;k,l,r,\alpha)\right|\ll \alpha^{\epsilon}|k|^{\epsilon}l^{1/2+\epsilon}(l,k_2^2)^{1/2}
\end{eqnarray}
due to Soundararajan (\cite{S2}, Lemma 5.3). \\

The locations of the poles and the growth on vertical lines of $\mathcal{C}_{\psi}(s)$ and $\mathcal{S}_{\psi}(s)$ are determined by the Dirichlet $L$-functions in the denominator of (\ref{denom}). In particular, since $\psi$ is a non-quadratic character, at most one of the characters 
\begin{eqnarray}
\psi(\cdot)\varphi(\cdot)\left(\frac{k_1r/2}{\cdot}\right)\hspace{0.5cm}\textrm{or}\hspace{0.5cm}\overline\psi(\cdot)\varphi(\cdot)\left(\frac{k_1r/2}{\cdot}\right)\nonumber
\end{eqnarray}
can be principal, so $\mathcal{D}_{\psi,\varphi}(s;k,l,\alpha)$ is meromorphic 
 in the half plane $\sigma>1/2$ with at most one pole of order $1$ at $s=1$. Assuming that $r$ is even and squarefree, we shall now show that such poles exist if and only if $k_1=1$, i.e that the character
 \begin{eqnarray}\label{character}
\psi (\cdot)\left(\frac{k_1r/2}{\cdot}\right)\hspace{2cm}
\end{eqnarray}
has modulus $r$  if and only if $k_1=1$. Since $\psi$ is a character modulo $q|r$, a solution exists when 
 \begin{eqnarray}\label{firstfind}
\varphi_0(\cdot)\left(\frac{k_1r/2}{\cdot}\right)
\end{eqnarray}
is a character modulo $r$. This is the case when $k_1=1$ because then the Kronecker symbol is a primitive character modulo $r/2$ since $r$ is squarefree. On the other hand, for any integer $k_1\neq 1$ we have
\begin{eqnarray}\label{char1}
\left(    \frac{k_1r/2}{\cdot}\right)=\left(    \frac{(k_1,r/2)}{\cdot}\right)^2\left(    \frac{\frac{k_1r/2}{(k_1,r/2)^2}}{\cdot}\right)
\end{eqnarray}
so, if (\ref{firstfind}) is to be a character modulo $r$, we certainly require that $(k_1r/2)/(k_1,r/2)^2$ divides $r$. In other words, we require that $k_1=(k_1,r/2)^2$ or $k_1=2(k_1,r/2)^2$, but neither of these are fundamental discriminants. This proves that the character in (\ref{character}) is a character modulo $r$ if and only if $k=\square$, 
and so $\mathcal{D}_{\psi,\varphi}(s;k^2,l,r,\alpha)$ has a simple pole at $s=1$ precisely when $\varphi=\varphi_0\overline\psi\chi_{r/2}$ and $\varphi=\varphi_0\psi\chi_{r/2}$, thus
\begin{eqnarray}\label{cossimp0}
\mathcal{D}_{\psi,\varphi_0\overline\psi\chi_{r/2}}(s;k^2,l,r,\alpha)&=&\mathcal{G}_{\psi,\varphi_0\overline\psi\chi_{r/2}}(s;k^2,l,r,\alpha)L(s, \varphi_0\overline\psi^2)L(s, \varphi_0)\nonumber\\
&=&\mathcal{G}_{\psi,\varphi_0\overline\psi\chi_{r/2}}(s;k^2,l,r,\alpha)L\left(s, \varphi_0\overline\psi^2\right)\zeta(s)\prod_{p|r}\left(1-\frac{1}{p^s}\right),\nonumber\\
\end{eqnarray}
and similarly with $\psi$ replaced with $\overline\psi$. Also, from the definitions (\ref{dddef}) and (\ref{euler}), we observe that $\mathcal{D}_{\psi,\varphi_0\overline\psi\chi_{r/2}}(s;k^2,l,r,\alpha)$ and $\mathcal{G}_{\psi,\varphi_0\overline\psi\chi_{r/2}}(s;k^2,l,r,\alpha)$ depend only on $\psi^2$ since
\begin{eqnarray}\label{note1}
\psi(n)d_{\psi}(n)=\sum_{d|n}\overline\psi^2(d),
\end{eqnarray}
so from now on we suppress the double subscript notation and define
\begin{eqnarray}\label{note2}
\mathcal{G}_{\psi^2}(s;k^2,l,r,\alpha)=\mathcal{G}_{\overline\psi,\varphi_0\psi\chi_{r/2}}(s;k^2,l,r,\alpha). 
\end{eqnarray}
Thus, in this notation, the equality (\ref{cossimp0}) is
\begin{eqnarray}\label{cossimp}
\mathcal{D}_{\psi^2}(s;k^2,l,r,\alpha)=\mathcal{G}_{\psi^2}(s;k^2,l,r,\alpha)L\left(s, \varphi_0\psi^2\right)\zeta(s)\prod_{p|r}\left(1-\frac{1}{p^s}\right).
\end{eqnarray}\\

Moving the path of integration to the line $\epsilon-1/2$ for any fixed $0<\epsilon<1/2$ (the bounds that we obtain in Section \ref{Rev} justify this), by Cauchy's theorem the integral (\ref{melinv}) is equal to
\begin{eqnarray}\label{melinvtwo}
&&\textrm{Res}\left(    \mathcal{C}_{\psi}(s+1)    f\left(\frac{kmX}{\alpha^2lr},s\right)+\mathcal{S}_{\psi}(s+1)  f\left(-\frac{kmX}{\alpha^2lr},s\right):s=0      \right) \nonumber\\
&+& \frac{1}{2\pi i}\int_{(\epsilon-1/2)}\left(    \mathcal{C}_{\psi}(s+1)    f\left(\frac{kmX}{\alpha^2lr},s\right)+\mathcal{S}_{\psi}(s+1)  f\left(-\frac{kmX}{\alpha^2lr},s\right)      \right)      \frac{ds}{m^s}.\nonumber\\
\end{eqnarray}
The preceding analysis shows that the residue in (\ref{melinvtwo}) is non-zero precisely when $k=\square$. Making the change of variables $k\mapsto k^2$ and using (\ref{cossimp}), the residue is
\begin{eqnarray}\label{nondiag2}
&&\frac{2}{r} \Re\Biggr(    \left\langle \cos\left(\frac{2\pi k^2\cdot}{r}\right)f\left(\frac{k^2mX}{\alpha^2lr},0\right)+\sin\left(\frac{2\pi k^2\cdot}{r}\right) f\left(-\frac{k^2mX}{\alpha^2lr},0\right) ,\varphi_0\overline\psi\chi_{r/2}\right\rangle    \nonumber\\
&\times& \varphi_0\psi\chi_{r/2}\left(\overline h\alpha^2l\right)\mathcal{G}_{\overline\psi^2}(1;k^2,l,r,\alpha)L\left(1, \varphi_0\overline\psi^2\right)\Biggr)\nonumber\\
\end{eqnarray}
and we note that (\ref{nondiag2}) is independent of $X$ because, using the definition (\ref{Fdef}), we have
\begin{eqnarray}\label{inter}
f\left(\frac{k^2mX}{\alpha^2lr},0\right)&=&\int_{0}^{\infty}\int_{1}^{2}\Phi(y)\omega_2\left(\frac{t}{8qXy}\right)\textrm{cas}\left(\frac{2\pi k^2mXy}{\alpha^2 lr t}\right)dy         \frac{dt}{t}\nonumber\\
&=&\widehat\Phi(0)\int_{0}^{\infty}\omega_2\left(\frac{1}{t}\right)\textrm{cas}\left(\frac{2\pi k^2mt}{8\alpha^2 lqr}\right)      \frac{dt}{t}
\end{eqnarray}
by changing variables $t\mapsto 8qXy/t$, and similarly with the first argument of $f$ negated. We now write 
\begin{eqnarray}\label{omegadef2}
\Omega(y)=\frac{1}{y}\omega_2\left(\frac{1}{y}\right)
\end{eqnarray}
so that 
\begin{eqnarray}\label{indx}
f\left(\frac{k^2mX}{\alpha^2lr},0\right)=\widehat\Phi(0)\widetilde \Omega\left(\frac{k^2m}{8\alpha^2lqr}\right)\nonumber
\end{eqnarray}
and the residue in (\ref{nondiag2}) is
\begin{eqnarray}\label{nondiag7}
&&\frac{2\widehat\Phi(0)}{r}  \Re\Biggr(   \left\langle \cos\left(\frac{2\pi k^2\cdot}{r}\right)\widetilde \Omega\left(\frac{k^2mX}{\alpha^2lr}\right)+\sin\left(\frac{2\pi k^2\cdot}{r}\right) \widetilde \Omega\left(-\frac{k^2mX}{\alpha^2lr}\right) ,\varphi_0\overline\psi\chi_{r/2}\right\rangle    \nonumber\\
&\times& \varphi_0\psi\chi_{r/2}\left(\overline h\alpha^2l\right)\mathcal{G}_{\overline\psi^2}(1;k^2,l,r,\alpha)L\left(1, \varphi_0\overline\psi^2\right)\Biggr).\nonumber\\
\end{eqnarray}

From (\ref{secmndiag}) and (\ref{nondiag7}),  we conclude that $\mathcal{R}_0=\mathcal{P}_2+\mathcal{R}$ where
\begin{eqnarray}\label{nondiag3}
\mathcal{P}_2&=&\frac{4\widehat\Phi(0)\varphi_0(l)}{lr^2}\sum_{m\in\mathcal{S}(r)}\frac{\chi_{8h}(m)d_\psi(m)}{m^{1/2}}\sum_{\substack{\alpha\leq Y\\(\alpha,lr)=1\\}}\frac{\mu(\alpha)}{\alpha^2}\sum_{k\neq 0}\nonumber\\
&\times&\Re\Biggr(
 \left\langle \cos\left(\frac{2\pi k^2\cdot}{r}\right)\widetilde\Omega\left(\frac{k^2m}{8\alpha^2lqr}\right)+\sin\left(\frac{2\pi k^2\cdot}{r}\right) \widetilde\Omega\left(-\frac{k^2m}{8\alpha^2lqr}\right),\varphi_0\overline\psi\chi_{r/2}\right\rangle         \nonumber\\
&\times&\psi\chi_{r/2}\left(\overline h\alpha^2l\right)\mathcal{G}_{\overline\psi^2}(1;k^2,l,r,\alpha)L\left(1, \varphi_0\overline\psi^2 \right)\Biggr)\nonumber\\
\end{eqnarray}
is the non-diagonal main term and, for $0<\epsilon<1/2$,
\begin{eqnarray}\label{Rdef}
\mathcal{R}&=&2\sum_{k\neq 0}\sum_{m\in\mathcal{S}(r)} \sum_{\substack{\alpha\leq Y\\(\alpha,l)=1}}\sum_{\varphi\pmod r} \frac{\chi_{8h}(m)d_\psi(m)\mu(\alpha)\varphi(\alpha^2l)}{m^{1/2}  \alpha^2 lr\phi(r) \pi i}\int_{(\epsilon-1/2)}   \mathcal{D}_{\psi, \varphi}(w+1;k,l,r,\alpha)\nonumber\\
&\times &\left(  \left\langle \cos\left(\frac{2\pi kh\cdot}{r}\right) ,\overline\varphi  \right\rangle  f\left(\frac{kmX}{\alpha^2 lr},w\right) 
+\left\langle \sin\left(\frac{2\pi kh\cdot}{r}\right) ,\overline\varphi  \right\rangle f\left(-\frac{kmX}{\alpha^2 lr},w\right)       \right)    \frac{ dw}{m^{w}}\nonumber\\
\end{eqnarray}
are the remaining non-diagonal terms, which we bound in Section \ref{Rev} below. \\

We now give an integral representation of $\mathcal{P}_2$.

\begin{lemma}For any $0<\epsilon <1/2$ there is a constant $\mathcal{N}_{h,l,r,\psi,\Phi}$ such that
\begin{eqnarray}\label{Rbound4}
\mathcal{P}_2=\mathcal{N}_{h,l,r,\psi,\Phi}+O_{\epsilon,\Phi}\left(l^{\epsilon}r^{\epsilon-2}Y^{\epsilon-1}\right)\nonumber
\end{eqnarray}
given by
\begin{eqnarray}\label{almost}
\mathcal{N}_{h,l,r,\psi,\Phi}&=&\frac{4\widehat\Phi(0)\varphi_0(l)}{lr^2}\Re\Biggr(     \frac{L\left(1, \varphi_0\overline\psi^2 \right) \tau\left(\overline\psi\chi_{r/2}\right)\psi\chi_{r/2}\left(\overline {2h}l\right)}{2\pi i}
\nonumber\\
&\times&\int_{(\epsilon)}\frac{\Gamma^2\left(\frac{s}{2}+\frac{1}{4}\right)\Gamma_1(s)\mathcal{K}_{\overline\psi^2}(s;l,r)}{\Gamma^2\left(\frac{1}{4}\right)A_{r/q}\left(s+\frac{1}{2},\chi_{8h}\psi\right)A_{r/q}\left(s+\frac{1}{2},\chi_{8h}\overline\psi \right)}\left(\frac{8lqr}{\pi}\right)^s\frac{ds}{s}\Biggr)\nonumber
\end{eqnarray}
where we have defined
\begin{eqnarray}
\Gamma_1(s)=(2\pi)^{-s}\Gamma(s)\textrm{\emph{cas}}\left(\frac{\pi s}{2}\right)\nonumber
\end{eqnarray}
and
\begin{eqnarray}\label{keuler5}
\mathcal{K}_{\chi}(s;l,r)&=&\left(\frac{\chi(2)}{2^{2s-1}}-1\right)L\left(2s,\chi\right)L\left(2s+1,\chi^3\right)A_{r/q}\left(2s+1,\chi^3\right)\mathcal{E}_{\chi}(s;l,r),\nonumber
\end{eqnarray}
where
\begin{eqnarray}
\mathcal{E}_{\chi}(s;l,r)&=&
\prod_{\substack{p|l_1}}p^{\delta-1/2}\frac{\chi^{(\delta-1)/2}(p)}{p^{(\delta-1)s}}\left(1-\frac{1}{p}\right)\left(1-\frac{\chi(p)}{p}\right)\left(1+\frac{\chi^2(p)}{p^{2s}}\right)\nonumber\\
&\times&\prod_{\substack{p|l\\p\nmid l_1}}\frac{\chi^{\delta}(p)}{p^{(\delta-1)s}}\left(1-\frac{1}{p^{}}\right)\left(1-\frac{\chi^2(p)}{p^2}\right)
\nonumber\\
&\times&\prod_{p\nmid lr}\left(1-\frac{1}{p^{}}\right)\left(1-\frac{\chi(p)}{p^{}}\right)\left(1+\frac{1+\chi(p)}{p} +\frac{\chi^2(p)}{p^3} -\frac{\overline\chi(p)}{p^2}\left(\frac{\chi^4(p)}{p^{2s}}+p^{2s}    \right)   \right)\nonumber\\
\end{eqnarray}
is holomorphic and bounded for fixed $-1/2<\sigma<1/2$.
 \end{lemma}

\begin{proof}
Using the Chinese remainder theorem and the fact that $\overline\psi\chi_{r/2}$ is the even primitive character that induces the character $\varphi_0\overline\psi\chi_{r/2}$, we begin by computing the inner product
\begin{eqnarray}\label{gausssum2}
\left\langle  e\left(\pm \frac{k^2\cdot}{r}\right),   \varphi_0\overline\psi\chi_{r/2}  \right\rangle &=& e\left(\frac{k^2}{2}\right)\sum_{a \pmod {r/2}}\overline\psi\chi_{r/2}(2a)e\left(\frac{ak^2}{r/2}\right)\nonumber\\
&=&(-1)^k\overline\psi\chi_{r/2}\left(2k^2\right)\tau\left(\overline\psi\chi_{r/2}\right)
\end{eqnarray}
where $\tau$ is the Gauss sum. By  (\ref{gausssum2}), the display in the second and third lines of (\ref{nondiag3}) is the real part of
\begin{eqnarray}\label{nondiag4}
L\left(1, \varphi_0\overline\psi^2 \right) \tau\left(\overline\psi\chi_{r/2}\right)\psi\chi_{r/2}\left(\overline {2h}\alpha^2\overline k^2l\right)(-1)^k\mathcal{G}_{\overline\psi^2}(1;k^2,l,r,\alpha)\widetilde\Omega\left(\frac{k^2m}{8\alpha^2lqr}\right),\nonumber\\
\end{eqnarray}
thus
\begin{eqnarray}\label{nondiagfour}
\mathcal{P}_2&=&\frac{4\widehat\Phi(0)\varphi_0(l)}{lr^2}\Re\Biggr(     L\left(1, \varphi_0\overline\psi^2 \right) \tau\left(\overline\psi\chi_{r/2}\right)\psi\chi_{r/2}\left(\overline {2h}l\right)\sum_{m\in\mathcal{S}(r)}\frac{\chi_{8h}(m)d_\psi(m)}{m^{1/2}}
\nonumber\\
&\times&\sum_{\substack{\alpha\leq Y\\(\alpha,lr)=1}}\frac{\mu(\alpha)\psi^2\left(\alpha\right)}{\alpha^2}\sum_{k=1}^{\infty}(-1)^k\overline\psi\left(k^2\right)\mathcal{G}_{\overline\psi^2}(1;k^2,l,r,\alpha)\widetilde\Omega\left(\frac{k^2m}{8\alpha^2lqr}\right)\Biggr).\nonumber\\
\end{eqnarray}
\\

Definition \ref{def1} and (\ref{omegadef2}) allow us to write $\widetilde\Omega$ in terms of $\Gamma$ and $\Gamma_1$ by Mellin inversion (interchanging the order of summation and integration is justified due to the rapid decay of the $\Gamma$-function on vertical lines and Lemma \ref{xibounds}) and, using (\ref{adef}) to express the sum over $m$ in terms of $A_{r/q}$, the triple sum in (\ref{nondiagfour}) is equal to 
\begin{eqnarray}\label{nondiagint}
&&\sum_{\substack{\alpha\leq Y\\(\alpha,lr)=1}}\frac{\mu(\alpha)\psi^2\left(\alpha\right)}{\alpha^2}\nonumber\\
&\times& \frac{1}{2\pi i}\int_{(1/2<c<1)}\frac{\Gamma^2\left(\frac{s}{2}+\frac{1}{4}\right)\Gamma_1(s)\mathcal{H}_{\overline\psi^2}(s;l,r,\alpha)}{\Gamma^2\left(\frac{1}{4}\right)A_{r/q}\left(s+\frac{1}{2},\chi_{8h}\psi\right)A_{r/q}\left(s+\frac{1}{2},\chi_{8h}\overline\psi \right)}\left(\frac{8\alpha^2lqr}{\pi}\right)^s\frac{ds}{s}\nonumber\\
\end{eqnarray}
where we have put
\begin{eqnarray}
\mathcal{H}_{\chi}(s;l,r,\alpha)=\sum_{k=1}^{\infty}\frac{(-1)^k\chi\left(k\right)}{k^{2s}}\mathcal{G}_{\chi}(1;k^2,l,r,\alpha),\nonumber
\end{eqnarray}
which is absolutely convergent for $\sigma>1/2$ by (\ref{Gbound}). \\

 To complete the sum over $\alpha$ we move the path of integration so that $0<c<1/2$, which we now justify. By Lemma \ref{gausssum} note that 
 \begin{eqnarray}
 \mathcal{G}_{\chi,p}(s;k^2)= \mathcal{G}_{\chi,p}(s;p^{2\gamma})\nonumber
 \end{eqnarray}
where $k^2=\prod p^{2\gamma}$, so
\begin{eqnarray}
\mathcal{H}_{\chi}(s;l,r,\alpha)=\left(\frac{\chi(2)}{2^{2s-1}}-1\right)\mathcal{H}^*_{\chi}(s;l,r,\alpha)\nonumber
\end{eqnarray}
where we have put
\begin{eqnarray}
\mathcal{H}^*_{\chi}(s;l,r,\alpha)=\prod_p\mathcal{H}^*_{\chi,p}(s)\nonumber
\end{eqnarray}
in which 
\begin{eqnarray}\label{heuler}
\mathcal{H}^*_{\chi,p}(s)=\sum_{\gamma=0}^{\infty}\frac{\chi\left(p^{\gamma}\right)}{p^{2\gamma s}}\mathcal{G}_{\chi,p}(1;p^{2\gamma}).
\end{eqnarray}
The Euler factors $\mathcal{H}^*_{\chi,p}(s)$ in (\ref{heuler}) above are distinguished by the six mutually exclusive cases $p|q$, $p|\frac{r}{q}$, $p|\alpha$, $p\nmid \alpha lr$, $p|l_1$ and $p|l$ but $p\nmid l_1$. These are 

\begin{eqnarray}\label{holotwo}
   = \begin{cases}
    1 &  \text{if }p| q,
    \\
    \\
    \left(1-\frac{\chi(p)}{p^{2s}}\right)^{-1} & \text{if }p|\frac{r}{q} ,
    \\
     \\
     \left(1-\frac{1}{p}\right)\left(1-\frac{\chi(p)}{p}\right)\left(1-\frac{\chi(p)}{p^{2s}}\right)^{-1} & \text{if }p|\alpha ,
     \\
      \\
     \left(1-\frac{\chi(p)}{p^{2s}}\right)^{-1}\left(1-\frac{\chi^3(p)}{p^{2s+1}}\right)^{-1}\left(1-\frac{1}{p^{}}\right)\left(1-\frac{\chi(p)}{p^{}}\right)\left(1+\frac{1+\chi(p)}{p}-\frac{\chi^3(p)}{p^{2s+2}}\right) & \text{if }p\nmid \alpha lr.
     \\
     \\
     p^{\delta-1/2}\frac{\chi^{(\delta-1)/2}(p)}{p^{(\delta-1)s}}\left(1-\frac{\chi(p)}{p^{2s}}\right)^{-1}\left(1-\frac{\chi^3(p)}{p^{2s+1}}\right)^{-1}\left(1-\frac{1}{p}\right)\left(1-\frac{\chi(p)}{p}\right)\left(1+\frac{\chi^2(p)}{p^{2s}}\right)
      & \text{if $p|l_1$},
     \\
     \\
     \frac{\chi^{\delta}(p)}{p^{(\delta-1)s}}\left(1-\frac{\chi(p)}{p^{2s}}\right)^{-1}\left(1-\frac{\chi^3(p)}{p^{2s+1}}\right)^{-1}\left(1-\frac{1}{p^{}}\right)\left(1-\frac{\chi^2(p)}{p^2}\right)&\text{if $p|l$, but $p\nmid l_1$}.
    \end{cases}\nonumber\\
\end{eqnarray}
To prove (\ref{holotwo}), note that the cases $p|\alpha r$ are immediate from (\ref{holo}), (\ref{note1}) and (\ref{note2}). For the cases $p|l$ and  $p\nmid \alpha lr$ we use Lemma \ref{gausssum} and (\ref{note1}). In the latter case the factors are
\begin{eqnarray} \label{efaccomp}
&&\left(1-\frac{1}{p}\right)\left(1-\frac{\chi(p)}{p}\right)\sum_{\gamma=0}^{\infty}\frac{\chi^{\gamma}\left(p\right)}{p^{2\gamma s}}
      \sum_{\beta=0}^{\infty}
      \frac{1}{p^{\beta}}\sum_{\delta=0}^{\beta} \chi^{\delta}(p)\frac{G_{p^{2\gamma}}(p^{\beta})}{p^{\beta/2}}\nonumber\\
&=&\left(1-\frac{1}{p}\right)\left(1-\frac{ \chi(p)}{p}\right)\sum_{\gamma=0}^{\infty}\frac{ \chi^{\gamma}(p)}{p^{2\gamma s}}\nonumber\\
&\times&
      \left(1+\left(1-\frac{1}{p}\right)\sum_{\beta=1}^{\gamma}\frac{1}{p^{\beta }}\sum_{\delta=0}^{2\beta} \chi^{\delta}(p)
      +    \frac{1}{p^{\gamma+1}}\sum_{\beta=0}^{2\gamma+1} \chi^{\beta}(p)
     \right).
     \end{eqnarray}
Using the identity 
\begin{eqnarray}\label{identity}
&&1+\left(1-y\right)\sum_{\beta=1}^{\gamma}y^{\beta}\sum_{\delta=0}^{2\beta} x^{\delta}+y^{\gamma+1}\sum_{\beta=0}^{2\gamma+1}x^{\beta}=
y+\frac{1+xy}{1-x^2y}\left(1-\left(x^2y\right)^{\gamma+1}\right)\nonumber
\end{eqnarray}
the display in  (\ref{efaccomp}) is equal to 
\begin{eqnarray}\label{eulerderiv}
&&\left(1-\frac{1}{p}\right)\left(1-\frac{\chi(p)}{p}\right)\Biggr(1+\frac{1+\chi(p)}{p}\nonumber\\
&+&\sum_{\gamma=1}^{\infty}\frac{\chi(p)}{p^{2\gamma s}}\left(\frac{1}{p}+\frac{1+\frac{\chi(p)}{p}}{1-\frac{\chi^2(p)}{p}}-\frac{1+\frac{\chi(p)}{p}}{1-\frac{\chi^2(p)}{p}}\left(\frac{\chi^2(p)}{p}\right)^{\gamma+1}\right)\Biggr)\nonumber\\
&=&
\left(1-\frac{1}{p}\right)\left(1-\frac{\chi(p)}{p}\right)\Biggr(1+\frac{1+\chi(p)}{p}\nonumber\\
&+&\left(\frac{1}{p}+\frac{1+\frac{\chi(p)}{p}}{1-\frac{\chi^2(p)}{p}}\right)\left(\frac{1}{1-\frac{\chi(p)}{p^{2s}}}-1\right)-\frac{\chi^2(p)}{p}\frac{1+\frac{\chi(p)}{p}}{1-\frac{\chi^2(p)}{p}}\left(\frac{1}{1-\frac{\chi^3(p)}{p^{2s+1}}}-1\right)\Biggr)
\nonumber\\
&=&\left(1-\frac{\chi(p)}{p^{2s}}\right)^{-1}\left(1-\frac{\chi^3(p)}{p^{2s+1}}\right)^{-1}\left(1-\frac{1}{p^{}}\right)\left(1-\frac{\chi(p)}{p^{}}\right)\left(1+\frac{1+\chi(p)}{p}-\frac{\chi^3(p)}{p^{2s+2}}\right) \nonumber\\
\end{eqnarray}
which proves the $p\nmid \alpha lr$ case in (\ref{holotwo}). The $p|l$ cases may be proved in a similar manner. \\

It follows that the function 
\begin{eqnarray}\label{facbound}
\frac{\mathcal{H}_{\chi}^*(s;l,r,\alpha)}{L\left(2s,\chi\right)L\left(2s+1,\chi^3\right)}\nonumber
\end{eqnarray} 
 is holomorphic and bounded for every fixed $\sigma> -1/2$ and so the rapid decay of the $\Gamma$ functions in the integrand  in (\ref{nondiagint}) permits us to move the path of integration to any vertical line with $-1/2<c<1$. Moving to the line $c=\epsilon$ for any $0<\epsilon<1/2$ we may complete the sum over $\alpha$ introducing an error that is $\ll l^{\epsilon}r^{\epsilon-2}Y^{\epsilon-1}$. Thus \begin{eqnarray}\label{Rbound4}
\mathcal{P}_2=\mathcal{N}_{h,l,r,\psi,\Phi}+O\left(l^{\epsilon}r^{\epsilon-2}Y^{\epsilon-1}\right)
\end{eqnarray}
where
\begin{eqnarray}\label{almost}
\mathcal{N}_{h,l,r,\psi,\Phi}&=&\frac{4\widehat\Phi(0)\varphi_0(l)}{lr^2}\Re\Biggr(     \frac{L\left(1, \varphi_0\overline\psi^2 \right) \tau\left(\overline\psi\chi_{r/2}\right)\psi\chi_{r/2}\left(\overline {2h}l\right)}{2\pi i}
\nonumber\\
&\times&\int_{(\epsilon)}\frac{\Gamma^2\left(\frac{s}{2}+\frac{1}{4}\right)\Gamma_1(s)\mathcal{K}_{\overline\psi^2}(s;l,r)}{\Gamma^2\left(\frac{1}{4}\right)A_{r/q}\left(s+\frac{1}{2},\chi_{8h}\psi\right)A_{r/q}\left(s+\frac{1}{2},\chi_{8h}\overline\psi \right)}\left(\frac{8lqr}{\pi}\right)^s\frac{ds}{s}\Biggr),\nonumber\\
\end{eqnarray}
so
\begin{eqnarray}
\mathcal{K}_{\chi}(s;l,r)=\left(\frac{\chi(2)}{2^{2s-1}}-1\right)\mathcal{K}^*_{\chi}(s;l,r)\nonumber
\end{eqnarray}
where we have put
\begin{eqnarray}\label{keuler}
\mathcal{K}_{\chi}^*(s;l,r)=\sum_{\substack{\alpha\geq 1\\(\alpha,lr)=1}}\frac{\mu(\alpha)\chi\left(\alpha\right)}{\alpha^{2-2s}} \mathcal{H}^*_{\chi}(s;l,r,\alpha).
\end{eqnarray}

To derive the factorisation of $\mathcal{K}_{\chi}(s;l,r)$ given in (\ref{keuler5}), we use (\ref{heuler}) and (\ref{holotwo}) to write $\mathcal{K}_{\chi}^*(s;l,r)$ as an Euler product. Firstly, using  (\ref{heuler}), we observe that
\begin{eqnarray}\label{keuler2}
\mathcal{K}_{\chi}^*(s;l,r)=\prod_{p|lr}\mathcal{H}^*_{\chi,p}(s)\sum_{\substack{\alpha\geq 1\\(\alpha,lr)=1}}\frac{\mu(\alpha)\chi\left(\alpha\right)}{\alpha^{2-2s}} \prod_{p\nmid lr\alpha}\mathcal{H}^*_{\chi,p}(s)\prod_{p|\alpha}\mathcal{H}^*_{\chi,p}(s)
\end{eqnarray}
and, using (\ref{holotwo}), that the double product inside the summation on the right hand side of (\ref{keuler2}) is 

\begin{eqnarray}\label{keuler3}
&&\prod_{p\nmid lr}\left(1-\frac{\chi(p)}{p^{2s}}\right)^{-1}\left(1-\frac{\chi^3(p)}{p^{2s+1}}\right)^{-1}\left(1-\frac{1}{p^{}}\right)\left(1-\frac{\chi(p)}{p^{}}\right)\nonumber\\
&\times&\prod_{p\nmid lr\alpha}\left(1+\frac{1+\chi(p)}{p}-\frac{\chi^3(p)}{p^{2s+2}}\right)\prod_{p|\alpha}\left(1-\frac{\chi^2(p)}{p^{2s+1}}\right).
\end{eqnarray}
Next, observe that the second line of (\ref{keuler3}) is equal to
\begin{eqnarray}\label{keuler4}
\prod_{p\nmid lr}\left(1+\frac{1+\chi(p)}{p}-\frac{\chi^3(p)}{p^{2s+2}}\right)\prod_{p|\alpha}\frac{1-\frac{\chi^2(p)}{p^{2s+1}}}{1+\frac{1+\chi(p)}{p}-\frac{\chi^3(p)}{p^{2s+2}}},\nonumber
\end{eqnarray}
so the right hand side of (\ref{keuler2}) is
\begin{eqnarray}\label{keuler6}
&&L\left(2s,\chi\right)L\left(2s+1,\chi^3\right)\prod_{p|\frac{lr}{q}}\mathcal{H}^*_{\chi,p}(s)\left(1-\frac{\chi(p)}{p^{2s}}\right)\left(1-\frac{\chi^3(p)}{p^{2s+1}}\right)\nonumber\\
&\times&\prod_{p\nmid lr}\left(1-\frac{1}{p^{}}\right)\left(1-\frac{\chi(p)}{p^{}}\right)\left(1+\frac{1+\chi(p)}{p}-\frac{\chi^3(p)}{p^{2s+2}}\right)\nonumber\\
&\times&\sum_{\substack{\alpha\geq 1\\(\alpha,lr)=1}}\frac{\mu(\alpha)\overline\chi(\alpha)}{\alpha^{2-2s}}\prod_{p|\alpha}\frac{1-\frac{\chi^2(p)}{p^{2s+1}}}{1+\frac{1+\chi(p)}{p}-\frac{\chi^3(p)}{p^{2s+2}}}.
\nonumber\\
&=&L\left(2s,\chi\right)L\left(2s+1,\chi^3\right)A_{r/q}\left(2s+1,\chi^3\right)\prod_{p|l}\mathcal{H}^*_{\chi,p}(s)\left(1-\frac{\chi(p)}{p^{2s}}\right)\left(1-\frac{\chi^3(p)}{p^{2s+1}}\right)\nonumber\\
&\times&\prod_{p\nmid lr}\left(1-\frac{1}{p^{}}\right)\left(1-\frac{\chi(p)}{p^{}}\right)\left(1+\frac{1+\chi(p)}{p} +\frac{\chi^2(p)}{p^3} -\frac{\overline\chi(p)}{p^2}\left(\frac{\chi^4(p)}{p^{2s}}+p^{2s}    \right)   \right)\nonumber\\
\end{eqnarray}
and we note that the Euler product in (\ref{keuler4}) is holomorphic and bounded for any fixed  $-1/2<\sigma<1/2$. To complete the proof we use (\ref{holotwo}) to express the $p|l$ cases. 
\end{proof}

\subsubsection{Bounds on the remaining non-diagonal terms}\label{Rev}
In this section we show that the sums and integral defining $\mathcal{R}$ in (\ref{Rdef}) are absolutely convergent and find a suitable bound for $\mathcal{R}$. 
Firstly, from the definition (\ref{secmndiag}) and Lemma \ref{xibounds}, we observe that the terms with $n>X^{1+\epsilon}$ are exponentially decreasing in $X$, so  the $k$th term in  (\ref{secmndiag}) is 
\begin{eqnarray}
\ll    \left|\frac{Y^2lrX^{\epsilon}    }{k}  \right|^{\nu}\nonumber
\end{eqnarray}
for any fixed $\nu>0$. Since we are also assuming that $Y^2lr\ll X^{1/2-\delta}$ for some fixed $\delta>0$, we may ignore the terms with $|k|\geq X^{1/2-\delta}$ in (\ref{Rdef}) (with a new $\delta$), which introduces an error that is smaller than any fixed negative power of $X$. Thus
\begin{eqnarray}\label{Rdef2}
|\mathcal{R}|&=&\sum_{0<|k|<X^{1/2-\delta}}\sum_{m\in\mathcal{S}(r)} \sum_{\substack{\alpha\leq Y\\(\alpha,l)=1}}\sum_{\varphi\pmod r} \frac{2\chi_{8h}(m)d_\psi(m)\mu(\alpha)\varphi(\alpha^2l)}{m^{1/2}  \alpha^2 lr\phi(r) \pi i}\int_{(\epsilon-1/2)}   \mathcal{D}_{\psi, \varphi}(w+1;k,l,r,\alpha)\nonumber\\
&\times &\left(  \left\langle \cos\left(\frac{2\pi kh\cdot}{r}\right) ,\overline\varphi  \right\rangle  f\left(\frac{kmX}{\alpha^2 lr},w\right) 
+\left\langle \sin\left(\frac{2\pi kh\cdot}{r}\right) ,\overline\varphi  \right\rangle f\left(-\frac{kmX}{\alpha^2 lr},w\right)       \right)    \frac{ dw}{m^{w}}\nonumber\\
&+&O(X^{-\nu})\nonumber\\
\end{eqnarray}
for any $\nu>0$. Soundararajan (\cite{S2}, Lemma 5.2) has shown that $f(\xi,s)$, $\xi\neq 0$, is a holomorphic function of $s$ in $\Re s>-1$, and in the region $-1<\Re s<1$ satisfies the bound 
 \begin{eqnarray}\label{soundbound}
|f(\xi,s)|\ll(1+|s|)^{-\Re s -1/2}\exp\left(-\frac{1}{10}\frac{\sqrt{|\xi|}}{\sqrt{X(1+|s|)}}\right)|\xi|^{\Re s}|\check\Phi(s)|.
\end{eqnarray}
Thus, by (\ref{holo}), (\ref{soundbound}), the Gauss sum bound $\left\langle e\left(kh\cdot/r\right),\overline\varphi  \right\rangle\ll r^{1/2}$ and the fact that $\check\Phi(w)\ll |w|^{-\nu}$ for any $\nu>0$ because $\Phi(w)$ is compactly supported, we have
\begin{eqnarray}\label{boundingR}
|\mathcal{R}|&\ll&\frac{r^{1/2+\epsilon}}{l^{1/2}X^{1/2-\epsilon}}\sum_{0<|k|<X^{1/2-\delta}}\sum_{\alpha\leq Y}\frac{1}{|k|^{1/2-\epsilon}\alpha}\max_{\varphi  (\textrm{mod } r)}\int_{(1/2+\epsilon)}\left| \mathcal{G}_{\psi,\varphi}(s;k,l,\alpha)\right|\nonumber\\
&\times&\left|L\left(s,\psi\varphi\left(\frac{k_1r/2}{\cdot}\right)\right)L\left(s,\overline\psi\varphi\left(\frac{k_1r/2}{\cdot}\right)\right)   \right|\exp\left(-\frac{1}{10|s|Y}\sqrt{\frac{|k|}{lr}}\right)\frac{|ds|}{|s|^{\nu}}. \nonumber\\
\end{eqnarray}
Using a convexity bound  on the path of integration in (\ref{boundingR}) such as $|L(s,\chi)|\leq C(q|s|)^{1/4}$ due to Kolesnik \cite{K}, for instance, and noting that the modulus of the characters in (\ref{boundingR}) is $< qk_1r^2$, it follows that (\ref{boundingR}) is
\begin{eqnarray}
&\ll&\frac{r^{3/2+\epsilon}}{l^{1/2}X^{1/2-\epsilon}}\sum_{0<|k|<X^{1/2-\delta}}\sum_{\alpha\leq Y}\frac{1}{|k|^{1/4-\epsilon}\alpha}\max_{\substack{\varphi  (\textrm{mod } r)\\\Re s=1/2+\epsilon}}\left| \mathcal{G}_{\psi,\varphi}(s;k,l,\alpha)\right|\nonumber\\
&\times&\int_{(1/2+\epsilon)}\exp\left(-\frac{\sqrt{|k|}}{10|s|Y\sqrt{lr}}\right)\frac{|ds|}{|s|^{\nu}}\nonumber
\end{eqnarray}
and, by (\ref{Gbound}), this is  
\begin{eqnarray}
&\ll&\frac{l^{1/2+\epsilon}r^{3/2+\epsilon}(Y^2lr)^{\nu/2}}{X^{1/2-\epsilon}}\sum_{0<k<X^{1/2-\delta}}\sum_{\alpha\leq Y}\frac{1}{k^{\nu/2+1/4-\epsilon}\alpha^{1-\epsilon}}\nonumber
\end{eqnarray}
for any $\nu>1$. Taking $\nu=1+\epsilon$, we conclude that 
\begin{eqnarray}\label{Rbound6}
|\mathcal{R}|\ll \frac{l^{1+\epsilon}r^{2+\epsilon}Y^{1+\epsilon}}{X^{3/8-\epsilon}}. 
\end{eqnarray}

\subsection{The diagonal term}\label{diagsm} One may begin the evaluation of the diagonal term $\mathcal{P}_1$ of $\mathcal{M}_2$ in the same way as the diagonal term $\mathcal{D}$ in the expected value $\mathcal{E}$ that we carried out in Section \ref{diagevm}. Here the analogue of (\ref{gcd}) is
\begin{eqnarray}\label{Rbound8}
\mathcal{P}_1=P_{h,l,r,X,\psi}+O\left(\frac{X^{\epsilon}}{l_1^{1/2}r^{1/2-\epsilon}Y}\right)
\end{eqnarray}
where 
\begin{eqnarray}\label{an02}
P_{h,l,r,X,\psi}=\frac{2\varphi_0(l_1)}{L(2,\varphi_0)l_1^{1/2}r}    \sum_{m\in\mathcal{S}(r)}\left(\frac{8h}{m}\right)\frac{d_{\psi}(m)}{m^{1/2}}
\sum_{1}^{\infty}\frac{\varphi_0(n)d_{\psi}(l_1n^2)\widehat F_{2,l_1mn^2}(0)}{n\prod_{p| ln}\left(1+p^{-1}\right)}\nonumber
\end{eqnarray}
and the error term in (\ref{Rbound8}) is easily seen because $d_{\psi}(m) \ll m^{\epsilon}$. Using Definition \ref{def1} to write $\widehat F_{2,l_1mn^2}(0)$ as a double integral and using (\ref{phimel}), we get
\begin{eqnarray}\label{ant}
P_{h,l,r,X,\psi}&=&\frac{\varphi_0(l_1)}{L(2,\varphi_0)l_1^{1/2}r} \nonumber\\
&\times& \frac{1}{2\pi i}    \int_{(c>0)}\frac{\Gamma^2\left(\frac{s}{2}+\frac{1}{4}\right)\check\Phi\left(s\right)\mathcal{L}_{\psi}(2s+1;l,r)}{\Gamma^2\left(\frac{1}{4}\right)A_{r/q}\left(s+\frac{1}{2},\chi_{8h}\psi\right)A_{r/q}\left(s+\frac{1}{2},\chi_{8h}\overline\psi\right)}\left(\frac{8qX}{\pi l_1}\right)^{s}\frac{ds}{s}
\nonumber\\
\end{eqnarray}
where 
\begin{eqnarray}
\mathcal{L}_{\psi}(s;l,r)=\sum_{1}^{\infty}\frac{\varphi_0(n)d_{\psi}(l_1n^2)}{n^{s}\prod_{p| ln}\left(1+p^{-1}\right)}\nonumber
\end{eqnarray}
is absolutely convergent for $\sigma>1$.\\

Next, we note that Soundararajan (\cite{S2}, Lemma 5.1) has computed the Euler product for $\mathcal{L}_{\psi}(s;l,2q)$; the difference in our case is just the omission of more Euler factors when $2q$ is a proper factor of $r$, thus 
\begin{eqnarray}\label{final}
\mathcal{L}_{\psi}(s;l,r)=d_{\psi}(l_1)\zeta(s)L(s,\psi^2)L(s,\overline\psi^2)\eta_{\psi}(s;l,r)
\end{eqnarray}
where 
\begin{eqnarray}
\eta_{\psi}(s;l,r)=\prod_p\eta_{\psi,p}(s)\nonumber
\end{eqnarray}
is absolutely convergent for $\sigma>1/2$ and the Euler factors are given by
\begin{eqnarray}\label{holothree}
 \eta_{\psi,p}(s) = \begin{cases}
    1-\frac{1}{p^{s}} &  \text{if }p| q,
    \\\\
    \left(1-\frac{1}{p^{s}}\right) \left(1-\frac{\psi(p^2)}{p^{s}}\right) \left(1-\frac{\overline\psi(p^2)}{p^{s}}\right) & \text{if }p|\frac{r}{q} ,
    \\ 
     \\
     \frac{1}{1+p^{-1}}\left(1-\frac{1}{p^{s}}\right) & \text{if }p|l_1,
     \\
     \\
     \frac{1}{1+p^{-1}}\left(1-\frac{1}{p^{2s}}\right) & \text{if }p|l\textrm{ but }p\nmid l_1,
     \\
     \\
     1-\frac{d_{\psi}^2(p)}{p^s(p+1)} \left(1-\frac{1}{p^{s}}\right)+\frac{1}{p^s(p+1)}-\frac{1}{p^{2s}}-\frac{1}{p^{2s}(p+1)} & \text{if }p\nmid lr.
    \end{cases}\nonumber\\
\end{eqnarray}
Due to the rapid decay of the $\Gamma$ function and $\check\Phi$ on vertical lines, we may move the path of integration in (\ref{ant}) to the line $c=\epsilon-1/4$ for some fixed $0<\epsilon<1/4$ so that  (\ref{ant}) is equal to 
\begin{eqnarray}\label{ant2}
&&\frac{\varphi_0(l_1)d_{\psi}(l_1)}{L(2,\varphi_0)l_1^{1/2}r} \nonumber\\
&\times& \textrm{Res}\left(\frac{\Gamma^2\left(\frac{s}{2}+\frac{1}{4}\right)\check\Phi\left(s\right)\zeta(2s+1)L(2s+1,\psi^2)L(2s+1,\overline\psi^2)\eta_{\psi}(2s+1;l,r)}{\Gamma^2\left(\frac{1}{4}\right)A_{r/q}\left(s+\frac{1}{2},\chi_{8h}\psi\right)A_{r/q}\left(s+\frac{1}{2},\chi_{8h}\overline\psi\right)s}\left(\frac{8qX}{\pi l_1}\right)^{s}: s=0\right)
\nonumber\\
&+&O\left(\frac{l^{\epsilon-1/2}r^{\epsilon-1}}{X^{1/4-\epsilon}}\right).\nonumber\\
\end{eqnarray}
Combining the error terms (\ref{Rbound2}), (\ref{Rbound4}), (\ref{Rbound6}), (\ref{Rbound8}) and (\ref{ant2}), we obtain the error term (\ref{th2errorbound}) of Theorem \ref{th2}. Finally, to get the diagonal main terms, we evaluate the residue in (\ref{ant2}). Since 
\begin{eqnarray}
\frac{\zeta(2s+1)}{s}=\frac{1}{2s^2}+\frac{\gamma}{s}+O(1)\nonumber
\end{eqnarray}
 as $s\rightarrow 1$, the residue is 
\begin{eqnarray}
&&\frac{\widehat\Phi\left(0\right)\left|L(1,\psi^2)\right|^2\eta_{\psi}(1;l,r)}{\left|A_{r/q}\left(\frac{1}{2},\chi_{8h}\psi\right)\right|^2}\left(\frac{\log X}{2}+\gamma\right)
\nonumber\\
&+& \frac{1}{2}\frac{d}{ds}\frac{\Gamma^2\left(\frac{s}{2}+\frac{1}{4}\right)\check\Phi\left(s\right)L(2s+1,\psi^2)L(2s+1,\overline\psi^2)\eta_{\psi}(2s+1;l,r)}{\Gamma^2\left(\frac{1}{4}\right)A_{r/q}\left(s+\frac{1}{2},\chi_{8h}\psi\right)A_{r/q}\left(s+\frac{1}{2},\chi_{8h}\overline\psi\right)}\left(\frac{8q}{\pi l_1}\right)^{s}\biggr\vert _{s=0},\nonumber
\end{eqnarray}
which completes our proof of Theorem \ref{th2}.

\subsection{The non-diagonal main term when $\psi$ is quartic}\label{quarticderiv}
For simplicity take $l=1$, $r=2q$ and let $\psi$ be a fixed even primitive non-quadratic quartic character  so $\psi^4=\psi_0$ and $\psi^2\neq \psi_0$. The integral representation in (\ref{almost}) reduces to 
\begin{eqnarray}\label{almostspecial}
&&\frac{\widehat\Phi(0)}{q^2}\Re\Biggr( \frac{ L\left(1, \varphi_0\overline\psi^2 \right) \tau\left(\overline\psi\right)\psi\left(\overline {2h}\right)}{2\pi i}\int_{(\epsilon)}\frac{\Gamma^2\left(\frac{s}{2}+\frac{1}{4}\right)\Gamma_1(s)\mathcal{K}_{\overline\psi^2}(s;1,2)}{\Gamma^2\left(\frac{1}{4}\right)}\left(\frac{16q}{\pi}\right)^s\frac{ds}{s}\Biggr).\nonumber\\
\end{eqnarray}
By (\ref{keuler5}), the fact that $\psi^2=\psi^6$ is real and $\psi^4=\psi^8=\psi_0$, we have
\begin{eqnarray}\label{keuler6}
\mathcal{K}_{\overline\psi^2}(s;1,2)=2^{1-2s}\psi^2(2)L\left(2s,\psi^2\right)L\left(2s+1,\psi^2\right)\left(1-\frac{\psi^2(2)}{2^{1+2s}}\right)\left(1-\frac{\psi^2(2)}{2^{1-2s}}\right)\mathcal{E}_{\psi^2}(s;1,2q)\nonumber
\end{eqnarray}
where
\begin{eqnarray}
\mathcal{E}_{\psi^2}(s;1,2q)=\prod_{p\textrm{ odd}}\left(1-\frac{1}{p^{}}\right)\left(1-\frac{\psi(p^2)}{p^{}}\right)\left(1+\frac{1+\psi(p^2)}{p} +\frac{1}{p^3} -\frac{\psi(p^2)}{p^2}\left(\frac{1}{p^{2s}}+p^{2s}    \right)   \right)\nonumber
\end{eqnarray}
so the integrand in (\ref{almostspecial}) is
\begin{eqnarray}\label{integrand}
\frac{\mathcal{J}_{\psi^2}(s;1,2q)}{s}&=&\frac{2\psi^2(2)}{\Gamma^2\left(\frac{1}{4}\right)s}\left(1-\frac{\psi^2(2)}{2^{1+2s}}\right)\left(1-\frac{\psi^2(2)}{2^{1-2s}}\right)\mathcal{E}_{\psi^2}(s;1,2q)\nonumber\\
&\times&\left(\frac{4q}{\pi}\right)^s\Gamma^2\left(\frac{s}{2}+\frac{1}{4}\right)\Gamma_1(s)L\left(2s,\psi^2\right)L\left(2s+1,\psi^2\right).
\end{eqnarray}
In the case $\psi=1$, Soundararajan \cite{S2} noticed that the functional equation of $\zeta(s)$ and the fact that $\mathcal{E}_{1}(s;l,2q)=\mathcal{E}_{1}(-s;l,2q)$ imply that $\mathcal{J}_{1}(s;l,2q)$ is an even function of $s$, i.e. $\mathcal{J}_{1}(s;l,2q)=\mathcal{J}_{1}(-s;l,2q)$, so Cauchy's theorem implies that
\begin{eqnarray}
\frac{1}{2\pi i}\int_{(\epsilon)}\mathcal{J}_{1}(s;l,2q)\frac{ds}{s}=\frac{1}{2}\textrm{Res}\left(\frac{\mathcal{J}_{1}(s;l,2q)}{s}:s=0\right).\nonumber
\end{eqnarray}
For $l=1$ and a primitive character $\chi$ the functional equation of the $L$-function $L(s,\chi)$ implies that 
$\mathcal{J}_{\chi}(s;1,2q)=\tau^2(\chi)\mathcal{J}_{\overline\chi}(-s;1,2q)$.
In our case $\chi=\psi^2$ is real and non-principal so we have even symmetry, i.e 
 \begin{eqnarray}\label{ct}
\frac{1}{2\pi i}\int_{(\epsilon)}\mathcal{J}_{\psi^2}(s;1,2q)\frac{ds}{s}=\frac{1}{1+\tau^2\left(\psi^2\right)}\textrm{Res}\left(\frac{\mathcal{J}_{\psi^2}(s;1,2q)}{s}:s=0\right)\nonumber
\end{eqnarray}
and so (\ref{almostspecial}) is the real part of 
\begin{eqnarray}\label{almostspecial2}
&&\frac{ \widehat\Phi(0)L\left(1, \varphi_0\overline\psi^2 \right) \tau\left(\overline\psi\right)\psi\left(\overline {2h}\right)}{q^2\left(1+\tau^2\left(\psi^2\right)\right)}\textrm{Res}\left(\frac{\mathcal{J}_{\psi^2}(s;1,2q)}{s}:s=0\right)\nonumber
\end{eqnarray}
in which the pole at $s=0$ is of order $2$ due to the factor $\Gamma_1(s)/s$. Since 
\begin{eqnarray}
\frac{\Gamma_1(s)}{s}=\frac{\Gamma(s+1)}{(2\pi)^ss^2}+\frac{\pi\Gamma(s+1)}{2(2\pi)^ss}+O(1)\nonumber
\end{eqnarray}
as $s\rightarrow 0$, the residue is 
\begin{eqnarray}
&&\frac{d}{ds}\frac{\Gamma(s+1)\Gamma^2\left(\frac{s}{2}+\frac{1}{4}\right)\mathcal{K}_{\psi^2}(s;1,2q)}{\Gamma^2\left(\frac{1}{4}\right)}\left(\frac{8q}{\pi^2}\right)^s\biggr\vert _{s=0}+\frac{\pi}{2}\mathcal{K}_{\psi^2}(0;1,2)\nonumber\\
&=&\mathcal{K}_{\psi^2}'(0;1,2q)+\left(\gamma+\frac{2\Gamma'\left(\frac{1}{4}\right)}{\Gamma\left(\frac{1}{4}\right)}+\log\left(\frac{8q}{\pi^2}\right)+\frac{\pi}{2}\right)\mathcal{K}_{\psi^2}(0;1,2q).\nonumber
\end{eqnarray}

\subsection{The expected value of the non-diagonal terms}
We conclude by noting that the expected value of the non-diagonal terms $\mathcal{P}_2$ over the group $h\in\mathbb{Z}^{\times}_r$ is zero. From (\ref{nondiag3}) for instance, the expected value introduces a factor 
\begin{eqnarray}\label{charsum}
\sum_{\substack{h\pmod r\\(h,r)=1}}\overline\psi\chi_{r/2}\left(h\right)\left(\frac{h}{m}\right)
\end{eqnarray}
for each odd $m\in\mathcal{S}(r)$, in which the Kronecker symbol $(\frac{\cdot}{m})$ is a character of modulus $m$ induced by a character of modulus $m_0|r$. Since  $\overline\psi\chi_{r/2}$ is non-principal, (\ref{charsum}) is zero by character orthogonality, which proves our remark after the statement of Theorem \ref{th2}.

\vspace{0.5cm}

\noindent \textit{Acknowledgment.}  
The first author is grateful to the Leverhulme Trust (RPG-2017-320) for the support through the research project grant ``Moments of $L$-functions in Function Fields and Random Matrix Theory". The research of the second author is supported by a Ph.D. studentship from the College of Engineering, Mathematics and Physical Sciences at the University of Exeter.


\end{document}